\documentclass[11pt]{article}
\usepackage{fullpage} 

\usepackage{ucs}
\usepackage{amssymb}
\usepackage{amsthm}
\usepackage{amsmath}
\usepackage{latexsym}
\usepackage[cp1251]{inputenc}
\usepackage[english]{babel}
\usepackage{graphicx}
\usepackage{wrapfig}
\usepackage{caption}
\usepackage{subcaption}
\usepackage{txfonts}
\usepackage{lmodern}
\usepackage{mathrsfs}
\usepackage{algorithm}
\usepackage{dsfont}
\usepackage{enumerate}
\usepackage[hidelinks]{hyperref}
\usepackage{multicol}
\usepackage{csquotes}
\usepackage{stmaryrd}
\usepackage{tikz}
\usepackage{comment}
\usepackage{enumitem}

\usepackage{cite}

\newcommand{\vc}[1]{\ensuremath{\vcenter{\hbox{#1}}}}

\tikzset{vtx/.style={inner sep=1.7pt, outer sep=0pt, circle, fill,draw}}

\tikzset{vtxB/.style={inner sep=1.7pt, outer sep=0pt, circle, fill=white,draw}}

\tikzset{vtx/.style={inner sep=1.7pt, outer sep=0pt, circle, fill}} 
\tikzset{unlabeled_vertex/.style={inner sep=1.7pt, outer sep=0pt, circle, fill}} 
\tikzset{labeled_vertex/.style={inner sep=2.2pt, outer sep=0pt, rectangle, fill=yellow, draw=black}} 
\tikzset{edge_color0/.style={color=black,line width=1.2pt}} 
\tikzset{edge_color1/.style={color=red,  line width=1.2pt,opacity=0}} 
\tikzset{edge_color2/.style={color=blue, line width=1.2pt,opacity=1}} 
\tikzset{edge_color2d/.style={color=blue, line width=1.2pt,opacity=1,dashed}} 
\tikzset{edge_color3/.style={color=blue, line width=1.2pt,opacity=1,dashed}} 
\tikzset{edge_color4/.style={color=red,  line width=1.2pt,dotted}} 
\tikzset{edge_color5/.style={color=blue, line width=1.2pt,dotted}} 
\tikzset{edge_color6/.style={color=green, line width=1.2pt,dotted}} 
\tikzset{vertex_color1/.style={inner sep=1.7pt, outer sep=0pt, draw, circle, fill=red}} 
\tikzset{vertex_color2/.style={inner sep=1.7pt, outer sep=0pt, draw, circle, fill=blue}} 
\tikzset{vertex_color3/.style={inner sep=1.7pt, outer sep=0pt, draw, circle, fill=green}} 
\def\outercycle#1#2{ \draw \foreach \x in {0,1,...,#1}{(270-45+\x*360/#2:0.8) coordinate(x\x)};}
\def\outercycleX#1#2#3{ \draw \foreach \x in {0,1,...,#1}{(#3+\x*360/#2:0.8) coordinate(x\x)};}

\tikzset{vertex_u/.style={inner sep=1.7pt, outer sep=0pt, circle, fill}} 
\tikzset{vertex_l/.style={inner sep=2.2pt, outer sep=0pt, rectangle,fill=yellow, draw=black}}

\newcommand{\Fuu}[1]{
\vc{\begin{tikzpicture}[scale=0.6]\outercycle{3}{2}
\draw[edge_color#1] (x0)--(x1);
\draw (x0) node[unlabeled_vertex]{};\draw (x1) node[unlabeled_vertex]{};
\end{tikzpicture}}}

\newcommand{\Flu}[1]{
\vc{\begin{tikzpicture}[scale=0.6]\outercycle{3}{2}
\draw[edge_color#1] (x0)--(x1);
\draw (x0) node[labeled_vertex]{} node{\scalebox{0.5}{\tiny 1}};\draw (x1) node[unlabeled_vertex]{};
\end{tikzpicture}}}

\newcommand{\Fuuu}[3]{
\vc{\begin{tikzpicture}[scale=0.6]\outercycleX{4}{3}{270}
\draw[edge_color#1] (x0)--(x1);\draw[edge_color#2] (x0)--(x2);  \draw[edge_color#3] (x1)--(x2);    
\draw (x0) node[unlabeled_vertex]{};\draw (x1) node[unlabeled_vertex]{};\draw (x2) node[unlabeled_vertex]{};
\end{tikzpicture}}}

\newcommand{\Fluu}[3]{
\vc{\begin{tikzpicture}[scale=0.6]\outercycleX{4}{3}{270}
\draw[edge_color#1] (x0)--(x1);\draw[edge_color#2] (x0)--(x2);  \draw[edge_color#3] (x1)--(x2);    
\draw (x0) node[labeled_vertex]{} node{\scalebox{0.5}{\tiny 1}}; \draw (x1) node[unlabeled_vertex]{};\draw (x2) node[unlabeled_vertex]{};
\end{tikzpicture}}}

\newcommand{\Fllu}[3]{
\vc{\begin{tikzpicture}[scale=0.6]\outercycleX{4}{3}{210}
\draw[edge_color#1] (x0)--(x1);\draw[edge_color#2] (x0)--(x2);  \draw[edge_color#3] (x1)--(x2);    
\draw (x0) node[labeled_vertex]{} node{\scalebox{0.5}{\tiny 1}}; 
\draw (x1) node[labeled_vertex]{}  node{\scalebox{0.5}{\tiny 2}};
\draw (x2) node[unlabeled_vertex]{};
\end{tikzpicture}}}

\newcommand{\FfourEdges}[6]{
\draw[edge_color#1] (x0)--(x1);\draw[edge_color#2] (x0)--(x2);\draw[edge_color#3] (x0)--(x3);  \draw[edge_color#4] (x1)--(x2);\draw[edge_color#5] (x1)--(x3);  \draw[edge_color#6] (x2)--(x3);
}
\newcommand{\Ffour}[5]{
\vc{\begin{tikzpicture}[scale=0.7]\outercycle{5}{4}
\FfourEdges#5
\draw (x0) node[vertex_#1]{};\draw (x1) node[vertex_#2]{};\draw (x2) node[vertex_#3]{};\draw (x3) node[vertex_#4]{};
\end{tikzpicture}}
}

\newcommand{\Flluu}[6]{
\vc{\begin{tikzpicture}[scale=0.7]\outercycle{5}{4}
\FfourEdges{#1}{#2}{#3}{#4}{#5}{#6}
\draw (x0) node[vertex_l]{} node{\scalebox{0.5}{\tiny 1}}; 
\draw (x1) node[vertex_l]{} node{\scalebox{0.5}{\tiny 2}}; 
\draw (x2) node[vertex_u]{};
\draw (x3) node[vertex_u]{};
\end{tikzpicture}}
}

\newcommand{\Flllu}[6]{
\vc{\begin{tikzpicture}[scale=0.7]\outercycleX{5}{4}{180}
\FfourEdges{#1}{#2}{#3}{#4}{#5}{#6}
\draw (x0) node[vertex_l]{} node{\scalebox{0.5}{\tiny 1}}; 
\draw (x1) node[vertex_l]{} node{\scalebox{0.5}{\tiny 2}}; 
\draw (x2) node[vertex_l]{} node{\scalebox{0.5}{\tiny 3}}; 
\draw (x3) node[vertex_u]{};
\end{tikzpicture}}
}

\newcommand{\Fuuuu}[6]{\Ffour{u}{u}{u}{u}{#1#2#3#4#5#6}}

\newcommand{\Fllluu}[7]{
\vc{\begin{tikzpicture}[scale=0.7]\outercycleX{6}{5}{270-72}
\draw[edge_color2] (x0)--(x1);
\draw[edge_color1] (x0)--(x2);
\draw[edge_color#1] (x0)--(x3);  
\draw[edge_color#2] (x0)--(x4);  
\draw[edge_color1] (x1)--(x2);
\draw[edge_color#3] (x1)--(x3);  
\draw[edge_color#4] (x1)--(x4);  
\draw[edge_color#5] (x2)--(x3);
\draw[edge_color#6] (x2)--(x4);
\draw[edge_color#7] (x3)--(x4);

\draw (x0) node[vertex_l]{} node{\scalebox{0.5}{\tiny 1}}; 
\draw (x1) node[vertex_l]{} node{\scalebox{0.5}{\tiny 2}}; 
\draw (x2) node[vertex_l]{} node{\scalebox{0.5}{\tiny 3}};
\draw (x3) node[vertex_u]{};
\draw (x4) node[vertex_u]{};
\end{tikzpicture}}
}

\newtheorem{theo}{Theorem}
\newtheorem{prop}[theo]{Proposition}
\newtheorem{lemma}[theo]{Lemma}
\newtheorem{ques}[theo]{Question}
\newtheorem{corl}[theo]{Corollary}
\newtheorem{conj}[theo]{Conjecture}

\theoremstyle{definition}

\numberwithin{theo}{section}

\newcommand{\FC}[1]{\textcolor{blue}{#1}}

\newcommand{\oururl}{\url{http://lidicky.name/pub/10problems}}

\author{%
J\'ozsef Balogh\footnote{Department of Mathematics, University of Illinois at Urbana-Champaign, Urbana, Illinois 61801, USA, E-mail: \texttt{jobal@illinois.edu}. Research is partially supported by NSF Grant DMS-1764123, NSF RTG grant DMS 1937241, Arnold O. Beckman Research Award (UIUC Campus Research Board RB 22000), the Langan Scholar Fund (UIUC), and the Simons Fellowship.} 
\and Felix Christian Clemen \footnote {Department of Mathematics, University of Illinois at Urbana-Champaign, Urbana, Illinois 61801, USA, E-mail: \texttt{fclemen2@illinois.edu}.}
 \and Bernard Lidick\'{y} \footnote {Iowa State University, Department of Mathematics, Iowa State University, Ames, IA., E-mail: \texttt{ lidicky@} \texttt{iastate.edu}. Research of this author is partially supported by NSF grant DMS-1855653 and Scott Hanna fellowship.}
}

\title{10 Problems for Partitions of Triangle-free Graphs}
\begin{document}
\maketitle 
\begin{abstract}
    We will state 10 problems, and solve some of them, for partitions in triangle-free graphs related to Erd\H{o}s' Sparse Half Conjecture.
    
    Among others we prove the following variant of it: For every sufficiently large even integer $n$ the following holds. Every triangle-free graph on $n$ vertices has a partition $V(G)=A\cup B$ with $|A|=|B|=n/2$ such that $e(G[A])+e(G[B])\leq n^2/16$. This result is sharp since the complete bipartite graph with class sizes $3n/4$ and $n/4$ achieves equality, when $n$ is a multiple of 4.
    
    Additionally, we discuss similar problems for $K_4$-free graphs. 
\end{abstract}
\section{Introduction}
Erd\H{o}s' \emph{Sparse Half Conjecture}~\cite{MR0409246,MR1439273} states that every triangle-free graph on $n$ vertices has a subset of vertices of size $n/2$ spanning at most $n^2/50$ edges. He offered a \$250 reward for its solution. There has been a lot of work on this conjecture~\cite{MR1273598,MR2232396,sparsehalfraz,MR3383248}, with the most recent progress by Razborov~\cite{sparsehalfraz}, who proved that every triangle-free graph on $n$ vertices has a subset of vertices of size $n/2$ spanning at most $(27/1024)n^2$ edges.  

Another related conjecture of Erd\H{o}s~\cite{MR0409246} on triangle-free graphs states that every triangle-free graph on $n$ vertices can be made bipartite by removing at most $n^2/25$ edges. There also has been work on this conjecture~\cite{Howtotriangle,Howmany,flagbipartite}, with the most recent progress by the authors of this paper who proved that every triangle-free graph on $n$ vertices can be made bipartite by removing at most $n^2/23.5$ edges.    

In this paper we will state various related questions and conjectures on partition problems for triangle-free graphs. A vertex \emph{$k$-partition} of a graph $G$ is a partition $V(G)=A_1\cup\ldots \cup A_k$ of its vertex set into $k$ classes. A $k$-partition is \emph{balanced} if $|A_i|=|A_j|$ for all $i,j\in[k]$. Given a vertex-subset $A$, we denote by $e(G[A])$ the number of edges in $G$ with both endpoints from $A$. If $G$ is clear from context, we also write $e(A)$. Given two disjoint vertex-subsets $A,B$ we denote by $e(A,B)$ the number of edges with one endpoint from $A$ and the other from $B$.
Let 
\[
D_k^b(G) :=  \min_{A_1,\ldots,A_k} \sum_{i=1}^k e(G[A_i]),
\]
where the minimum is taken over all balanced $k$-partitions $V(G)=A_1\cup\ldots \cup A_k$. 
The edges with both endpoints from the same set $A_i$ are called \emph{class-edges}.

Our first result can be interpreted as a combined variant of both of Erd\H{o}s' conjectures. 
\begin{theo}
\label{theo:balancedcut}
There exists $n_0$ such that for every even integer $n\geq n_0$ the following holds. Let $G$ be a triangle-free graph $G$ on $n$ vertices. Then
\begin{align*}
    D_2^b(G)\leq \frac{n^2}{16}.
\end{align*}
Additionally, if $D_2^b(G)=\frac{n^2}{16}$, then $G\cong K_{\frac{3n}{4},\frac{n}{4}}$. 
\end{theo}
The problem of finding a balanced $2$-partition with fewest class-edges is known as MAX-BISECTION (see e.g. \cite{MR1367967,MR2228942}) in computer science. It is NP-complete with an easy reduction from MAX-CUT, the problem of finding a $2$-partition with fewest class-edges. The MAX-BISECTION problem has also received attention in combinatorics, see Bollob\'{a}s-Scott~\cite{MR1945378}, Alon-Bollob\'{a}s-Kri\-ve\-le\-vich-Sudakov~\cite{MR1983363}, Lee-Loh-Sudakov~\cite{MR3096334} and many others. One major difference compared to many other work in this area is that Theorem~\ref{theo:balancedcut} studies $D_2^b(G)$ as a function in the number of vertices instead of the number of edges.

Erd\H{o}s, Faudree, Rousseau and Schelp~\cite{MR1273598} proposed to study an unbalanced version of Erd\H{o}s' Sparse Half Conjecture, see Figure~\ref{fig:conj}(a) for an illustration.
\begin{conj}[Erd\H{o}s, Faudree, Rousseau and Schelp~\cite{MR1273598}]
\label{unbalconj2}
Let $G$ be an $n$-vertex graph and let $\alpha$ be an arbitrary constant, $53/120\leq \alpha\leq 1$. Further, let 
\begin{align*}
    \beta>\begin{cases} (2\alpha-1)/4, & \text{when } 17/30\leq \alpha\leq 1, \\
    (5\alpha-2)/25 & \text{when } 53/120\leq \alpha\leq 17/30. 
    \end{cases}
\end{align*}
If $n$ is sufficiently large and each $\lfloor \alpha n\rfloor$-subset of $V(G)$ spans at least $\beta n^2$ edges, then $G$ contains a triangle. 
\end{conj}

\begin{figure}[ht]
\begin{center}
    \begin{tikzpicture}[scale=8]
    \draw[->]      (0.35,0)--(1.05,0) node[below]{$\alpha$};
    \draw[->]      (0.4,0)--++(0,0.25) node[left]{$\beta$};
    \draw[thick]
     plot[domain=53/120:17/30,samples=40] (\x,{ (5*\x-2)/25})
     plot[domain=17/30:1,samples=40] (\x,{ (2*\x-1)/4})
    ;
    \draw 
    (53/120,0) -- ++(0,-0.03)node[below]{\small $\frac{53}{120}$}
    (17/30,0) -- ++(0,-0.03)node[below]{\small $\frac{17}{30}$}
    (0.6,0) -- ++(0,-0.03)node[below right]{\tiny $0.6$}
    (0.4,0) -- ++(0,-0.03)node[below left]{\tiny $0.4$}
    (1,0) -- ++(0,-0.03)node[below]{\tiny $1$}
    ;
    \draw (0.75,-0.1)node{(a)};
    \end{tikzpicture}
    \hskip 4em
    \begin{tikzpicture}[scale=8]
    \draw[->]      (0.45,0)--(1.05,0) node[below]{$\alpha$};
    \draw[->]      (0.5,0)--++(0,0.25) node[left]{$\beta$};
    \draw[thick]
     plot[domain=0.5857864377:1,samples=40]   (\x,{ (2*\x-1)/4})
     plot[domain=0.5:0.5857864377,samples=40] (\x,{ (1-\x)^2/4})
    ;
    \draw 
    (0.5857864377,0) -- ++(0,-0.03)node[below]{\tiny $2-\sqrt{2}$}
    (0.717,0) -- ++(0,-0.03)node[below]{\tiny $0.717$}
    (1,0) -- ++(0,-0.03)node[below]{\tiny $1$} 
    (0.5,0) -- ++(0,-0.03)node[below]{\tiny $\frac{1}{2}$}
    ;
    \draw (0.75,-0.1)node{(b)};
    \end{tikzpicture}    
\end{center}
    \caption{Bounds on $\beta$ from (a) Conjecture~\ref{unbalconj2}  and (b) Conjecture~\ref{unbalconj}.}
    \label{fig:conj}
\end{figure}
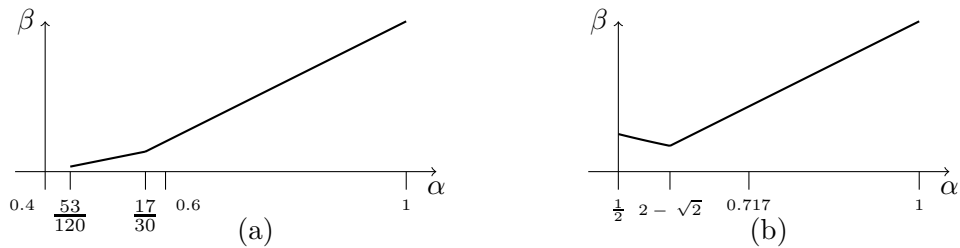

The complete balanced bipartite graph achieves equality in the first case and the balanced $C_5$-blowup achieves equality in the second case. Erd\H{o}s, Faudree, Rousseau and Schelp~\cite{MR1273598} proved Conjecture~\ref{unbalconj2} for $\alpha\geq 0.647$ and Krivelevich~\cite{MR1320169} for $\alpha\geq 0.6$, note that $17/30\approx 0.5666$, and $53/120\approx 0.4416$. Brandt~\cite{MR1606776} disproved Conjecture~\ref{unbalconj2} for $\alpha<0.474$.
He~\cite{MR1606776} observed that the Higman-Sims graph is a counterexample in this range. Yet, the conjecture is believed to be true for a wide range of $\alpha$. We extend the range for which Conjecture~\ref{unbalconj2} holds.
\begin{theo}
\label{theo: krivextension}
Conjecture~\ref{unbalconj2} holds for $\alpha\geq 0.579$.
\end{theo}

Motivated by Conjecture~\ref{unbalconj2}, we conjecture the following unbalanced version of Theorem~\ref{theo:balancedcut},
 see Figure~\ref{fig:conj}(b) for an illustration.
\begin{conj}
\label{unbalconj}
Let $G$ be an $n$-vertex graph and $\alpha$ be an arbitrary constant, $\frac{1}{2}\leq\alpha\leq1$. Further, let
\begin{align*}
    \beta> \begin{cases}
        (2\alpha-1)/4& \text{for } 2-\sqrt{2}\leq \alpha \leq 1, \\
    (1-\alpha)^2/4 & \text{for }\frac{1}{2}\leq \alpha< 2-\sqrt{2}. 
    \end{cases}
\end{align*}
If $n$ is sufficiently large and for each $\lfloor \alpha n\rfloor$-subset $A$ of $V(G)$ 
we have $e(A)+e(A^c)\geq \beta n^2$, then $G$ contains a triangle. 
\end{conj}
Note that Theorem~\ref{theo:balancedcut} verifies Conjecture~\ref{unbalconj} for $\alpha=1/2$.\footnote{We remark that numerical computations suggest that the proof of  Theorem~\ref{theo:balancedcut} could be extended to solve Conjecture~\ref{unbalconj} for $\alpha\in[0.5,0.5+\varepsilon]$ for some small $\varepsilon>0$.} If Conjecture~\ref{unbalconj} is true, then it is best possible. In the first case, it is shown by the complete balanced bipartite graph, and in the second case by the complete bipartite graph with class sizes $\lfloor(n+\alpha n)/2\rfloor$ and $\lceil(n-\alpha n)/2\rceil$. We prove Conjecture~\ref{unbalconj} for some range of $\alpha$. 
\begin{theo}
\label{unbaltheo}
Let $\alpha \geq 0.717$. Then 
Conjecture~\ref{unbalconj}
holds.
\end{theo}
 Note that Theorem~\ref{unbaltheo} strengthens Erd\H{o}s, Faudree, Rousseau and Schelp's~\cite{MR1273598} result on Conjecture~\ref{unbalconj2} for $\alpha \geq 0.717$. We also prove a variant of Theorem~\ref{theo:balancedcut} for balanced $3$-partitions. 
\begin{theo}
\label{theo:baltri}
There exists $n_0$ such that for every integer $n\geq n_0$ and divisible by 3, the following holds. Let $G$ be a triangle-free graph on $n$ vertices. Then 
\begin{align*}
    D_3^b(G)\leq \frac{n^2}{36}.
\end{align*}
Additionally, if $D_3^b(G)=\frac{n^2}{36}$, then $G\cong K_{\frac{5n}{6},\frac{n}{6}}$ or $G\cong K_{\frac{n}{2},\frac{n}{2}}$. 
\end{theo}
Theorems~\ref{theo:balancedcut}, \ref{unbaltheo} and \ref{theo:baltri} can be understood as $\ell_1$-norm results in the following sense. Given a balanced $k$-partition $V(G)=A_1\cup\ldots \cup A_k$, denote by $(e(A_i)_i)$ the vector containing $e(A_i)$ as entrices. Further, for $p\geq 1$ and $k\geq 2$, we define
\begin{align*}
    D_{k,p}^b(G)=\min \|e(A_i)_i\|_p,
\end{align*}
where the minimum is taken over all balanced $k$-partitions $V(G)=A_1\cup\ldots \cup A_k$ and $\|\cdot\|_p$ denotes the $\ell_p$-norm. Note that $D_{k,1}^b(G)=D_{k}^b(G)$.
\begin{ques}
Let $p\geq 1$ and $k\geq2$. What is the maximum of $D_{k,p}^b(G)$ over all $n$-vertex triangle-free graphs $G$?
\end{ques}
The following results answer this question for the $\ell_\infty$-norm and $k=2,3$.
\begin{theo}
\label{theo:maxbip}
Let $G$ be a triangle-free graph on $n$ vertices, where $n$ is sufficiently large and divisible by $12$. Then there exists a balanced $2$-partition such that each class spans at most $n^2/18$ edges, i.e. $D_{2,\infty}^b(G)\leq n^2/18$.
\end{theo}
Note that this theorem is best possible. The complete bipartite graph with class sizes $n/3$ and $2n/3$ has the property that any balanced 2-partition has one of the parts with at least $n^2/18$ edges. 
\begin{theo}
\label{theo:maxtrip}
Let $G$ be a triangle-free graph on $n$ vertices, where $n$ is divisible by 3. Then there exists a balanced $3$-partition of its vertex set such that the number of edges in each class is at most $n^2/48+o(n^2)$, i.e. $D_{3,\infty}^b(G)\leq n^2/48+o(n^2)$.
\end{theo}

This theorem is best possible up to the $o(n^2)$-error term. The balanced complete bipartite graph has the property that each of its balanced $3$-partitions has one of the parts with at least $n^2/48$ edges. 

Keevash and Sudakov~\cite{MR1967401} proved the following local density theorem for $K_{r+1}$-free graphs. 
\begin{theo}[Keevash, Sudakov~\cite{MR1967401}]
\label{SudakobKeevashclique}
Let $r\geq 2$ and let $G$ be a $K_{r+1}$-free graph on $n$ vertices. If $1-\frac{1}{2r^2}\leq \alpha\leq 1$ then $G$ contains a set of $\lfloor \alpha n\rfloor$ vertices spanning at most $\frac{r-1}{2r}(2\alpha-1)n^2$ edges.
\end{theo}
They~\cite{MR1967401} also conjectured that their theorem should hold for $\alpha\geq 1-1/r$ when $r\geq 3$. Using some of their arguments, we extend their result along the lines of Conjecture~\ref{unbalconj}.
\begin{theo}
\label{KeevashSudakovextension}
Let $r\geq 1$. There exists $c_r<1$ such that the following holds for $c_r\leq \alpha\leq 1$. Let $n$ be an integer such that $\alpha n$ is an integer. Then every $K_{r+1}$-free graph $G$ on $n$ vertices contains a set $A$ of $ \alpha n$ vertices such that 
\begin{align*}
    e(A)+e(A^c)\leq  \frac{r-1}{2r}(2\alpha-1)n^2.
\end{align*}
\end{theo}
Note that both Theorems~\ref{SudakobKeevashclique} and \ref{KeevashSudakovextension} are extensions of Tur\'an's theorem~\cite{Turanstheorem} and also both Theorems are sharp, where equality is achieved by the \emph{Tur\'an graph}. This is the graph constructed by partitioning a set of $n$ vertices into $r$ parts of sizes as equal as possible, and where two vertices are adjacent iff they are in different parts.

We further remark that all of our results on $K_{r+1}$-free graphs, namely Theorems~\ref{theo:balancedcut}, \ref{theo: krivextension}, \ref{unbaltheo}, \ref{theo:baltri}, \ref{theo:maxbip}, \ref{theo:maxtrip}, \ref{KeevashSudakovextension} can be extended to $H$-free graphs where $H$ has chromatic number $r+1$ when adding an additional error term of $o(n^2)$. This follows form a standard application of Szemer\'edi's Regularity Lemma, for example as done in \cite{MR1967401} or \cite{MR2359832}.

Some proofs in this paper use flag algebras to describe cuts in graph. This idea comes from Naves~\cite{navesgrwc} who proposed it at the Graduate Research Workshop in Combinatorics. It was also used by Hu, Lidick\'y, Martins, Norin, and Volec~\cite{flagmulti} and by the authors in~\cite{flagbipartite}.

Our paper is organized as follows. In Section~\ref{sec:openprob} we state various conjectures on related questions for triangle-free and $K_4$-free graphs. In Section~\ref{secmT1} we prove Theorem~\ref{theo:balancedcut}, in Section~\ref{sec:KrivextensionK3} we prove Theorem~\ref{theo: krivextension}, in Section~\ref{sec:baltri} we prove Theorem~\ref{theo:baltri}, in Section~\ref{sec:theounb} we prove Theorem~\ref{unbaltheo}, in Section~\ref{sec:clique-freegraphs} we prove our result on $K_{r+1}$-free graphs, Theorem~\ref{KeevashSudakovextension} and finally in Section~\ref{sec:maxkpart} we prove Theorems~\ref{theo:maxbip} and \ref{theo:maxtrip}. 

\section{Conjectures on related questions}
In this section we state various conjectures and questions on partition problems for triangle-free and $K_4$-free graphs. 
\label{sec:openprob}
\subsection{Triangle-free graphs}

Denote by $D_k(G)$ the minimum number of edges which can be deleted from $G$ to make it $k$-partite. 

\begin{conj}
Let $G$ be a triangle-free graph on $n$ vertices. Then  $D_3(G)\leq n^2/121$.
\end{conj}
If true, then this conjecture is sharp. Let $H$ be the balanced blow-up of the Gr\"otzsch graph. The \emph{Gr\"otzsch graph}, see Figure~\ref{fig:grotzsch}, has 11 vertices and 20 edges and has the fewest vertices among all triangle-free 4-chromatic graphs~\cite{MR0360330}. A result of Erd\H{o}s, Gy\H{o}ri and Simonovits~\cite[Theorem 7]{Howmany} states that there is a canonicial ``edge deletion'' achieving the minimum of $D_3(H)$. Every canonical edge deletion contains at least $(n/11)^2$ edges.

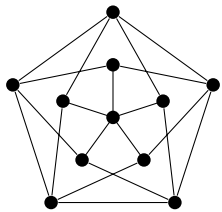
\begin{figure}[h!]
\begin{center}
\begin{tikzpicture}[scale=0.7]
\draw[every node/.style={inner sep=1.8pt,fill,circle}]
\foreach \x in {0,...,4}{ (0,0)--(90+72*\x:1) node(x\x){}  (90+72*\x:2) node(y\x){}}
(0,0) node{}
(y0)--(y1)--(y2)--(y3)--(y4)--(y0)
(y0)--(x1)--(y2)--(x3)--(y4)--(x0)--(y1)--(x2)--(y3)--(x4)--(y0)
;
\end{tikzpicture}
\end{center}
    \caption{The Gr\"otzsch graph.}
    \label{fig:grotzsch}
\end{figure}
All previous problems can be stated for $d$-regular graphs. While most of them become easier, a version of Theorem~\ref{theo:balancedcut} for regular graphs seems to increase the difficulty of the question.
\begin{ques}
What is the maximum of $D_2^b(G)$ over all $n$-vertex triangle-free regular graphs?
\end{ques}
The balanced $C_5$-blow-up seems to be a reasonable candidate for the extremal example when $10$ divides $n$.

\subsection{\texorpdfstring{$K_4$}\ -free graphs}
Many of the questions and conjectures stated in this paper can be asked for $K_4$-free graphs or more generally $H$-free graphs. In this section we present some conjectures on $K_4$-free graphs. 

A result by Sudakov~\cite{MR2359832} states that every $K_4$-free graph can be made bipartite by removing at most $n^2/9$ edges. This is sharp since the balanced complete 3-partite graph needs at least $n^2/9$ edges removed to make it bipartite. Sudakov~\cite{MR2359832} also conjectured a generalization of his result.   
\begin{conj}[Sudakov~\cite{MR2359832}]\label{conj:benny}
Fix $r\ge 3$. For every $n$-vertex $K_{r+1}$-free graph $G$, it holds that
\[
D_2(G) \leq \begin{cases} \frac{(r-1)^2}{4r^2}\cdot n^2   &\quad r\hbox{ odd, and} \\ \frac{r-2}{4r}\cdot n^2   &\quad r\hbox{ even.} \end{cases}
\]
\end{conj}
The conjecture was verified for $r=5$ by Hu, Lidick\'y, Martins, Norin, and Volec~\cite{flagmulti} using flag algebras. The conjecture for even $r$ seems to be more difficult than for odd $r$.

Recently Reiher~\cite{ReiherK4}, building up on work of Liu and Ma~\cite{LiuMaK4}, proved the corresponding sparse-half-version of this theorem: 
Every $K_4$-free graphs contains a set of size $n/2$ spanning at most $n^2/18$ edges. The following conjecture, if true, were to generalize both Sudakov's and Reiher's result. 
\begin{conj}
\label{conj:sparsedo}
Let $n$ be an even integer and $G$ be a $K_4$-free graph on $n$ vertices. Then there exists a balanced $2$-partition of its vertex set $V(G)=A\cup B$ such that $e(A)+e(B)\leq n^2/9$.
\end{conj}
If true, Conjecture~\ref{conj:sparsedo} is sharp, because the complete balanced $3$-partite graph satisfies $e(A)+e(B)\geq n^2/9$ for any balanced $2$-partition $V=A\cup B$.

Note that Reiher's result~\cite{ReiherK4} implies Conjecture~\ref{conj:sparsedo} for regular graphs. Further, it is also true for 3-partite graphs.
\begin{prop}
Conjecture~\ref{conj:sparsedo} holds for 3-partite graphs. 
\end{prop}
\begin{proof}
Let $G$ be an $n$-vertex $3$-partite graph with classes $V(G)=V_1\cup V_2\cup V_3$ and $|V_1|\geq |V_2|\geq |V_3|$. If $|V_1|\geq n/2$, then consider a balanced 2-partition $V=A\cup B$ such that $A\subseteq V_1$. Now, $e(A)=0$ and $e(B)\leq n^2/12$ by Tur\'an's theorem. Thus $e(A)+e(B)\leq n^2/12$. If $|V_1|\leq n/2$, then consider a balanced 2-partition $V=A\cup B$ such that $V_1\subseteq A\subseteq V_1\cup V_3$ and $V_2\subseteq B\subseteq V_2\cup V_3$. Now,
\begin{align*}
    e(A)\leq |V_1|\left(\frac{n}{2}-|V_1|\right) \quad \text{and} \quad 
        e(B)\leq |V_2|\left(\frac{n}{2}-|V_2|\right).
\end{align*}
Since $n/2\geq |V_1|\geq |V_2|\geq |V_3|$, we have $|V_1|+|V_2|\geq 2n/3$. If $5n/12\leq |V_1|\leq n/2$, then 
\begin{equation*}
    e(A)\leq |V_1|\left(\frac{n}{2}-|V_1|\right)\leq \frac{5n}{12}\cdot \frac{n}{12}=\frac{5n^2}{144} \quad \text{and} \quad  e(B)\leq |V_2|\left(\frac{n}{2}-|V_2|\right) \leq \frac{n^2}{16},
\end{equation*}
thus $e(A)+e(B)< n^2/9$. If $|V_1|\leq 5n/12$, then $|V_2|\geq 2n/3-|V_1|\geq n/4$ and therefore
\begin{align*}
    e(B)\leq |V_2|\left(\frac{n}{2}-|V_2|\right) \leq \left(\frac{2n}{3}-|V_1| \right)\left(\frac{n}{2}-\left(\frac{2n}{3}-|V_1|\right)\right)=-\frac{n^2}{9}-|V_1|^2+\frac{5n}{6}|V_1|,
\end{align*}
because the function $f(x):=x(1/2-x)$ is decreasing for $x\geq1/4$. Now,
\begin{align*}
    e(A)+e(B)&\leq |V_1|\left(\frac{n}{2}-|V_1|\right)-\frac{n^2}{9}-|V_1|^2+\frac{5n}{6}|V_1|=-\frac{n^2}{9}-2|V_1|^2+\frac{4n}{3}|V_1|\\
    &=\frac{n^2}{9}-2\left(|V_1|-\frac{n}{3}\right)^2\leq \frac{n^2}{9},
\end{align*}
completing the proof of this proposition.
\end{proof}

\begin{conj}
\label{K4bal2par}
Let $n$ be an even integer and $G$ be a $K_4$-free graph on $n$ vertices. Then there exists a balanced $2$-partition of $G$ such that each class spans at most $n^2/16$ edges.
\end{conj}
Conjecture~\ref{K4bal2par}, if true, is best possible, because the complete 3-partite graph with class sizes $n/2,n/4$ and $n/4$ has the property that every balanced 2-partition $V(G)=A\cup B$ satisfies \begin{align*}
    \max\{e(A),e(B)\}\geq \frac{n^2}{16}.
\end{align*}

\begin{prop}
Conjecture~\ref{K4bal2par} is true for 3-partite graphs. 
\end{prop}
\begin{proof}
Let $G$ be an $n$-vertex $3$-partite graph with classes $V(G)=V_1\cup V_2\cup V_3$. Then there exists $i\in[3]$ such that $|V_i|+|V_j|\geq n/2$ and $|V_i|+|V_k|\geq n/2$ for $j\neq k\in [3]\setminus \{i\}$. Thus, there exists a balanced 2-partition $V(G)=A\cup B$ with $V_j\subseteq A\subseteq V_i\cup V_j$ and $V_k \subseteq B\subseteq V_i\cup V_k$, therefore $e(A)\leq \frac{n^2}{16}$ and $e(B)\leq \frac{n^2}{16}$. 
\end{proof}

\begin{conj}
\label{K4bal3par}
Let $n$ be an integer divisible by 3 and $G$ be a $K_4$-free graph on $n$ vertices. Then there exists a balanced $3$-partition of its vertex set into three classes with at most $(4/81)n^2$ class-edges.
\end{conj}
Conjecture~\ref{K4bal3par}, if true, is best possible, because the complete 3-partite graph with class sizes $5n/9,2n/9$ and $2n/9$ achieves equality. 
\begin{prop}
Conjecture~\ref{K4bal3par} is true for 3-partite graphs. 
\end{prop}
\begin{proof}
Any 3-partite graph $G$ with vertex set $V$ contains an independent set $I$ of size $n/3$. The graph $G[V\setminus I]$ is 3-partite and since Conjecture~\ref{conj:sparsedo} holds for 3-partite graphs, there exists a balanced 2-partition $V\setminus I=A\cup B$ such that
\begin{align*}
    e(A)+e(B)\leq \frac{\left(\frac{2n}{3}\right)^2}{9}=\frac{4}{81}n^2.
\end{align*}
Hence, $V=I\cup A\cup B$ is a balanced 3-partition with at most $(4/81)n^2$ class-edges. 
\end{proof}

\section{Balanced 2-partition: proof of Theorem~\ref{theo:balancedcut}}
\label{secmT1}
We start our proof by observing that Razborov's~\cite{flagsRaz} result on the Erd\H{o}s' Sparse Half Conjecture implies Theorem~\ref{theo:balancedcut} for regular graphs. 
\begin{theo}[Razborov~\cite{flagsRaz}]
\label{Razprogr}
Every triangle-free graph on $n$ vertices contains a vertex set of size $\lfloor\frac{n}{2}\rfloor$ that spans at most $\frac{27}{1024}n^2$ edges.
\end{theo}
 Theorem~\ref{Razprogr} implies Theorem~\ref{theo:balancedcut} for regular graphs: Take a vertex set $A$ of size $n/2$ with $e(A)\leq (27/1024)n^2$. For regular graphs,  $e(A^c)=e(A)$. Thus, 
 \begin{align*}
     e(A)+e(A^c)\leq \frac{54}{1024}n^2 <  \frac{1}{16}n^2,
 \end{align*}
 implying Theorem~\ref{theo:balancedcut} for regular graphs. Next we prove Theorem~\ref{theo:balancedcut} for graphs which are almost regular with density about $1/3$. This is a subcase for which the main part of our proof fails. We handle it separately first. 
\begin{lemma}
\label{lem:almostreg}
Let $n$ be an even integer and $G$ be an $n$-vertex triangle-free graph such that 
\begin{align}
\label{flaglemmacutsup}
    \sum_{v\in V(G)}\left(\deg(v)-\frac{n}{3}\right)^2\leq 10^{-4}n^3. 
    \end{align}
Then $D_2^b(G)< \frac{1}{16}n^2$.
\end{lemma}
\begin{proof}
Let $\varepsilon=10^{-4}$ and $G$ be an $n$-vertex triangle-free graph such that \eqref{flaglemmacutsup} holds. For $S\subseteq V(G)$ with $|S|=n/2$ denote by $\bar{d}_S$ the average degree of the vertices in $S$, i.e.
\begin{align*}
    \bar{d}_S=\frac{2}{n}\sum_{x\in S}\deg(x).
\end{align*}
For a vertex set $S\subseteq V(G)$ of size $n/2$ we have
\begin{align}
\label{degsumcountd}
    2e(S)+e(S,S^c)=\sum_{x\in S}\deg(x)=\frac{n}{2}\bar{d}_S.
\end{align}
By the Cauchy-Schwarz inequality and the condition of this lemma, 
\begin{align*}
    \left(\frac{n}{2}\left(\bar{d}_S-\frac{n}{3}\right) \right)^2=\left(\sum_{x\in S}\left(\deg(x)-\frac{n}{3}\right) \right)^2\leq \frac{n}{2}\sum_{x\in S}\left(\deg(x)-\frac{n}{3}\right)^2\leq \frac{\varepsilon}{2} n^4.
\end{align*}
Thus $\left(\bar{d}_S-\frac{n}{3}\right)^2\leq 2\varepsilon n^2$
which implies $\left|\bar{d}_S-\frac{n}{3}\right|\leq \sqrt{2\varepsilon} n$. Since $S$ was an arbitrary set of size $n/2$, we also have
$\left|\bar{d}_{S^c}-\frac{n}{3}\right|\leq \sqrt{2\varepsilon} n$. Therefore, using the triangle inequality, we conclude  $\left|\bar{d}_{S}-\bar{d}_{S^c}\right|\leq 2\sqrt{2\varepsilon} n$. Now, by \eqref{degsumcountd} we have
\begin{align*}
    2e(S)+e(S,S^c)=\frac{n}{2}\bar{d}_{S}\leq \frac{n}{2}\left(\bar{d}_{S^c}+2\sqrt{2\varepsilon}n\right)=2e(S^c)+e(S,S^c)+\sqrt{2\varepsilon}n^2.
\end{align*}
This implies
\begin{align}
\label{edgesinhalf}
    e(S)\leq e(S^c)+\sqrt{\frac{\varepsilon}{2}}n^2.
\end{align}
By Theorem~\ref{Razprogr} there exists a vertex set $X\subseteq V(G)$ of size $n/2$ with $e(X)\leq \frac{27}{1024}n^2$. Applying \eqref{edgesinhalf} for $S=X^c$, we get
\begin{align*}
    D_2^b(G)\leq e(X)+e(X^c)\leq 2e(X)+\sqrt{\frac{\varepsilon}{2}}n^2\leq \frac{54}{1024}n^2+\sqrt{\frac{\varepsilon}{2}}n^2\leq \frac{n^2}{16},
\end{align*}
completing the proof of Lemma~\ref{lem:almostreg}.
\end{proof}

Next we prove Theorem~\ref{theo:balancedcut} for graphs with independence number at least $n/2$. 
Recall Mantel's theorem stats that every $n$-vertex triangle-free graph has at most $n^2/4$ edges, equality is achieved by $K_{\frac{n}2,\frac{n}2}$ for even $n$.
\begin{lemma}
\label{lem: ind}
Let $n$ be an even integer and $G$ be an $n$-vertex triangle-free graph satisfying $\alpha(G)\geq \frac{n}{2}$. Then $D_2^b(G)\leq \frac{n^2}{16}$.
Additionally, if $D_2^b(G)=\frac{n^2}{16}$, then $G\cong K_{\frac{3n}{4},\frac{n}{4}}$. 
\end{lemma}
\begin{proof}
Let $I$ be an independent set of size $n/2$. By Mantel's theorem,
\begin{equation}
    e(I)+e(V\setminus I)=e(V\setminus I)\leq \frac{|V\setminus I|^2}{4}=\frac{n^2}{16},
\end{equation}
proving the first part of this lemma. Let $G$ be an $n$-vertex triangle-free graph satisfying $\alpha(G)\geq \frac{n}{2}$ and $D_2^b(G)=\frac{n^2}{16}$. We can assume that $4$ divides $n$, as otherwise $n^2/16$ is not an integer. 
Let $A$ be an independent set of size $n/2$ and set $B:=V\setminus A$. We have
\begin{align*}
    e(B)=e(A)+e(B)\geq D_2^b(G)=\frac{n^2}{16},
\end{align*}
and thus, by Mantel's Theorem, $B$ spans a complete balanced bipartite graph, i.e. $B=B_1\cup B_2$ with $e(B_1)=e(B_2)=0$ and $e(B_1,B_2)=\frac{n^2}{16}$. Since $G$ is triangle-free, no vertex in $A$ can have neighbors in both $B_1$ and $B_2$. Therefore, without loss of generality, there exists a partition $A=A_1\cup A_2$ such that $|A_1|=|A_2|=n/4$ and $e(A_1,B_1)=0$. We have 
\begin{align*}
    e(A_2,B_2)=e(A_1\cup B_1)+e(A_2\cup B_2)\geq D_2^b(G)=\frac{n^2}{16},
\end{align*}
and thus $A_2\cup B_2$ spans a complete bipartite graph by Mantel's theorem. Since for $u\in B_2$ and $v\in B_1\cup A_2$, we have $uv\in E(G)$ and because $G$ is triangle-free, the set $B_1 \cup A_2$ is independent. We conclude
\begin{align*}
    e(A_1,B_2)=e(A_1\cup B_2)+e(A_2\cup B_1)\geq D_2^b(G)=\frac{n^2}{16},
\end{align*}
and thus $A_1\cup B_2$ also spans a complete bipartite graph by Mantel's theorem. We conclude that $G$ is a complete bipartite graph with classes $A_1\cup A_2\cup B_1$ and $B_2$, i.e. $G\cong K_{\frac{3n}{4},\frac{n}{4}}$.
\end{proof}

\begin{lemma}
\label{flaglemmacuts}
There exists $n_0$ such that for all $n\geq n_0$ the following holds. Let $G$ be an $n$-vertex triangle-free graph with $\alpha(G)\leq n/2$ and 
\begin{align}
\label{formula:flagalmoregular}
    \sum_{v\in V(G)}\left(\deg(v)-\frac{n}{3}\right)^2\geq 10^{-4}n^3, \quad
\end{align}
then $D_2^b(G)< \frac{1}{16}n^2$.
\end{lemma}
Note that Lemma~\ref{lem:almostreg} together with Lemma~\ref{flaglemmacuts} implies Theorem~\ref{theo:balancedcut}. We prove Lemma~\ref{flaglemmacuts} using the method of flag algebras in the next subsection.

\subsection{Introduction to flag algebras}


Razborov invented the method of flag algebra~\cite{flagsRaz} to derive bounds for parameters in extremal combinatorics with assistance by computers on optimization. Flag algebras have been applied in various settings such as graphs, hypergraphs, edge-coloured graphs, oriented graphs, permutations and discrete geometry.

Standard applications of the flag algebra method to graphs provide bounds on densities of induced subgraphs. To get bounds, inequalities and equalities involving the densities of induced subgraphs are combined with the help of semidefinite programming. This step is computer-assisted. We give a short introduction to the flag algebra notation in the setting of graphs.

Given two graphs $H$ and $G$, we denote by $d(H,G)$ the density of $H$ in $G$, i.e., the probability that a uniformly at random chosen $|V(H)|$-subset in $V(G)$ induces a copy of $H$. If $G$ is clear from context, then, abusing notation, we drop $G$ and just depict $H$ in place of $d(H,G)$. For example, let $G$ be a graph on $n$ vertices, then
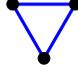
\begin{align*}
\Fuuu222
\end{align*}
denotes triangle density in $G$. Every inequality or equality we write in this notation holds with an error term $o(1)$ as $n \to \infty$, which we omit writting. 
Let $k$ be an integer and $\phi$ be an injective function $[k] \rightarrow V(G)$. In other words, $\phi$ labels $k$ distinct vertices in $G$. The pair $(G, \phi)$ is called a \emph{$k$-labeled graph}. 

Let $(H, \phi')$ and $(G, \phi)$ be two $k$-labeled graphs. Let $X$ be a uniformly at random chosen $(|V(H)|-k)$-subset of $V(G) \setminus Im(\phi)$. By $d((H,\phi'),(G,\phi))$ we denote the probability that the $k$-labeled subgraph of $G$ induced by $X$ and the $k$ labeled vertices, i.e., $(G[X \cup Im(\phi)], \phi)$, is isomorphic to $(H, \phi')$, where the isomorphism maps $\phi(i)$ to $\phi'(i)$ for $i \in [k]$. Again, if $(G,\phi)$ is clear from context, then, abusing notation, we drop it and just depict $(H,\phi')$, where the labeled vertices are marked as squares. For example, let $G$ be a graph on $n$ vertices with vertices $u$ and $v$ labeled with $1$ and $2$, respectively. Then, 
\begin{align*}
\Fllu222=\frac{|N(u)\cap N(v)|}{n-2}
\end{align*}
denotes the density of the common neighborhood of $u$ and $v$ in $G$. 

The linear \emph{averaging operator}
$\left\llbracket . \right\rrbracket$ denotes the average (or expectation) over all possible labelings.
Let $(H,\phi)$ be a $k$-labeled graph. 
We define $\left\llbracket (H,\phi)  \right\rrbracket = \alpha H$,
where $\alpha$ is the probability that a random injective map $\phi'$
from $[k]$ to $V(H)$ gives $(H,\phi')$ isomorphic to $(H,\phi)$. For example
\[
\left\llbracket\, \Flluu211212 \,\right\rrbracket  = \frac{2}{6}\, \Fuuuu211212.
\]
Extending $\left\llbracket . \right\rrbracket$ to linear combinations of labeled graphs is done by linearity.
The important feature of $\left\llbracket . \right\rrbracket$ is that if $X$ is a linear combination of labeled flags satisfying $X \geq 0$, then $\left\llbracket X  \right\rrbracket \geq 0$.

\subsection{Proof of Lemma~\ref{flaglemmacuts}}
Suppose $G$ is a triangle-free graph satisfying \eqref{formula:flagalmoregular}
and $D_2^b(G) \geq \frac{1}{16}n^2$. Since $G$ is triangle-free, we have
   \begin{align}
\label{triangle-free}
\Fuuu222 = 0.
\end{align}
Further, Lemma~\ref{lem: ind} implies $\Delta\leq n/2$, hence we have 
\begin{align}
\label{form: max degree}
\Flu2 \leq \frac{1}{2}.
\end{align}

Next, we find a balanced 2-partition based on the adjacencies of an arbitrary, but fixed vertex $v\in V(G)$. Denote by $\mathcal{A}\subset \mathcal{P}(n)$ the set of all sets $A$ of size $n/2$ containing $N(v)$.

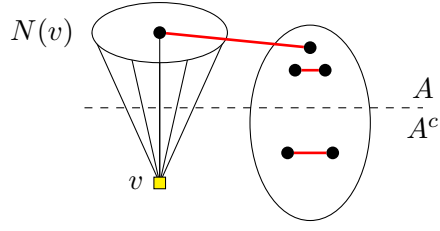
\begin{figure}
\begin{center}
    \begin{tikzpicture}
    \draw
    (0,0) node[labeled_vertex,label=left:$v$](v){}
    (-1,2) node[left]{$N(v)$}
    \foreach \x in {-2,-1,...,2.1}{
    (v)--(0.45*\x,2)
    }
    ;
    \draw[fill=white]
    (0,2) ellipse (0.9cm and 0.4cm) (2,0.8) ellipse (0.8cm and 1.3cm)
    ;
    \draw[dashed]
    (-1,1) -- (3.5,1) node[above]{$A$} node[below]{$A^c$}
    ;
    \draw 
    (0,2)node[vtx](a){}
    (2,1.8)node[vtx](b){}
    (1.8,1.5)node[vtx](c){}
    (2.2,1.5)node[vtx](d){}
    (1.7,0.4)node[vtx](e){}
    (2.3,0.4)node[vtx](f){}
    (v)--(a)
    ;
    \draw[red,line width=1pt]
    (a)--(b) (c)--(d) (e)--(f)
    ;
    \end{tikzpicture}
\end{center}
    \caption{Possible edges inside 2-partition from Lemma~\ref{flaglemmacuts}.}
    \label{fig:vcut}
\end{figure}

For every $A\in\mathcal{A}$, the 2-partition $A\cup A^c$ is balanced and therefore we have $e(A)+e(A^c)\geq n^2/16$ by assumption, see Figure~\ref{fig:vcut} for one such partition $A\cup A^c$. If $A\in \mathcal{A}$ were chosen uniformly at random, then a vertex in $N(v)^c$ is in $A$ with probability $\frac{\frac{n}{2}-\deg(x)}{n-\deg(x)}$. 
Since $N(v)$ is triangle-free, there are two kinds of edges in $A$. The first kind has one vertex in $N(v)$ and the other in $N(v)^c$. The other kind has both vertices in $N(v)^c$. These two kinds of edges are counted by 3-vertex flags in \eqref{eq:cut}. Notice the fraction multiplying the 3-vertex flags in \eqref{eq:cut} corresponds to $\frac{\frac{n}{2}-\deg(x)}{n-\deg(x)}$. Similar set-up appears in other proofs in this paper.
By averaging over all $A\in\mathcal{A}$, we have
\begin{align}
\frac{1}{\binom{n-1}{2}|\mathcal{A}|}\sum_{A\in\mathcal{A}}e(A)=\Fluu212 \frac{\left( \frac{1}{2} - \Flu2  \right)}{\Flu1} + \Fluu112 \left(\frac{\left( \frac{1}{2} - \Flu2  \right)}{\Flu1}\right)^2\label{eq:cut}
\end{align}
and
\[
\frac{1}{\binom{n-1}{2}|\mathcal{A}|}\sum_{A\in\mathcal{A}}e(A^c)= \Fluu112 \left(\frac{ \frac{1}{2}}{\Flu1}\right)^2 \quad.
\]
Since 
\[
\frac{1}{\binom{n-1}{2}|\mathcal{A}|}\sum_{A\in \mathcal{A}}e(A)+e(A^c)\geq \frac{n^2}{16\binom{n-1}{2}}=\frac{1}{8}+o(1),
\]
we get
\begin{align}
\label{ineq: flagineq1/8}
\frac{1}{8} \leq \Fluu212 \frac{\left( \frac{1}{2} - \Flu2  \right)}{\Flu1} + \Fluu112 \left(\frac{ \frac{1}{4}+\left( \frac{1}{2} - \Flu2  \right)^2}{\left( \Flu1 \right)^2}\right).
\end{align}
Inequality \eqref{ineq: flagineq1/8} can be rewritten as 
\begin{align}
\label{form: cut2}
0 \leq \Fluu212 \left( \frac{1}{2} - \Flu2\right)\Flu1  + \Fluu112 \left(\frac{1}{2}-\Flu2+\Fluu221 \right)
-
\frac{1}{8}\left( \Flu1 \right)^2 
. 
\end{align}
%
%
The condition \eqref{formula:flagalmoregular} stated in flag algebra notation is
\begin{align*}
\left\llbracket
    \left(\Flu2 - \frac{1}{3} \right)^2
    \right\rrbracket \geq 10^{-4}
           \quad \quad \text{implying} \quad \quad \left\llbracket
    \Fluu221 - \frac{2}{3} \Flu2+ \frac{1}{9} 
    \right\rrbracket \geq 10^{-4}.
\end{align*}
Applying the unlabeling operator gives
\begin{align}
\label{form: cut1}
     \frac{1}{3}\Fuuu221 - \frac{2}{3} \Fuu2+ \frac{1}{9} \geq 10^{-4}, \quad \text{because} \quad \left\llbracket  \Fluu221 \right\rrbracket= \frac{1}{3}\Fuuu221.
\end{align}
On the other hand, using \eqref{form: max degree} and \eqref{form: cut2}, computer assisted flag algebras calculation gives
\[
\frac{1}{3}\Fuuu221 - \frac{2}{3} \Fuu2+ \frac{1}{9} \leq 
10^{-7},
\]
contradicting \eqref{form: cut1}.
The files needed to perform the corresponding calculations are available at \oururl. 

\section{Local densities in triangle-free graphs: proof of Theorem~\ref{theo: krivextension}}
\label{sec:KrivextensionK3}
The following Lemma due to Erd\H{o}s, Faudree, Rousseau and Schelp~\cite{MR1273598} is helpful.
\begin{lemma}[Erd\H{o}s, Faudree, Rousseau and Schelp~\cite{MR1273598}]
\label{lem: indseterdos}
Let $\alpha$ be fixed, $0.5\leq\alpha\leq 1$ and $\beta>(2\alpha-1)/4$. Further, let $G$ be an $n$-vertex graph satisfying $\alpha(G)\geq (1-\alpha)n$. If $n$ is sufficiently large and each $\lfloor \alpha n\rfloor$-subset of $V(G)$ spans at least $\beta n^2$ edges, then $G$ contains a triangle.
\end{lemma}
We start by proving Theorem~\ref{theo: krivextension} for a finite set of values of $\alpha$. Our proof follows Krivelevich's proof~\cite{MR1320169} for the range $\alpha\geq 0.6$. The improvement stems from a flag algebra application.  
\begin{lemma}
\label{lem: krivextendisc}
Let $i\in\{0,1,\ldots,21\}$, $\alpha_i=0.579+0.0005i$ and $\beta_i=(2\alpha_i-1)/4-\frac{1}{4000}$. Further, let $G_i$ be an $n$-vertex graph satisfying $\alpha(G)< (1-\alpha_i)n$. If $n$ is sufficiently large and every $\lfloor \alpha n\rfloor$-subset of $V(G)$ spans at least $\beta_i n^2$ edges, then $G$ contains a triangle.
\end{lemma}
\begin{proof}
Let $i\in\{0,1,\ldots,21\}$ and $G_i$ be a triangle-free $n$-vertex graph satisfying $\alpha(G)< 1-\alpha_i$ and every $\lfloor \alpha n\rfloor$-subset of $V(G)$ spans at least $\beta_i n^2$ edges.
We have 
\begin{align}
    \label{flagcond1}
    \Fuuu222 = 0 \quad \quad \text{and} \quad \quad \Flu2 \leq 1-\alpha_i.
\end{align}
By averaging over all $\lfloor \alpha n\rfloor$-subsets of $V(G)$ containing $N(u)$ for a fixed vertex $u$, we have
\begin{align*}
    2\beta_i\leq \frac{1}{\binom{n-1}{2}\binom{n-1}{\lfloor \alpha n\rfloor}}\sum_{\substack{A\subseteq V(G_i)\\ N(u)\subseteq A \\|A|=\lfloor \alpha n\rfloor}}e(A)= \Fluu212 \left(\frac{\alpha_i-\Flu2}{\Flu1}\right) +\Fluu112 \left(\frac{\alpha_i-\Flu2}{\Flu1}\right)^2,
\end{align*}
implying
\begin{align}
\label{cutvertcutkriv}
   0 \leq \Fluu212 \left(\alpha_i-\Flu2\right)\Flu1 +\Fluu112 \left(\alpha_i-\Flu2\right)^2 -2\beta_i\left(\Flu1\right)^2.
\end{align}
By averaging over all $\lfloor \alpha n\rfloor$-subsets of $V(G)$ containing $N(u)$ and avoiding $N(v)$ for a fixed edge $uv$, we have
\begin{align*}
    2\beta_i\leq \frac{1}{\binom{n-2}{2}\binom{n-2}{\lfloor \alpha n\rfloor}}\sum_{\substack{A\subseteq V(G_i)\\ N(u)\subseteq A \\ N(v) \cap A = \emptyset  \\|A|=\lfloor \alpha n\rfloor}}e(A)= \Flluu212112 \left(\frac{\alpha_i-\Fllu221}{\Fllu211}\right) +\Flluu211112 \left(\frac{\alpha_i-\Fllu221}{\Fllu211}\right)^2,
\end{align*}
implying
\begin{align}
    \label{cutedgecutkriv}
    0\leq  \Flluu212112 \left(\alpha_i-\Fllu221\right)\Fllu211 +\Flluu211112 \left(\alpha_i-\Fllu221\right)^2
    -2\beta_i\left(\Fllu211\right)^2,
\end{align}
Applying the method of flag algebras shows there is no sufficiently large graph satisfying all three \eqref{flagcond1}, \eqref{cutvertcutkriv} and \eqref{cutedgecutkriv} at the same time, giving a contradiction. This part of the proof is computer assisted. The files needed to perform the corresponding calculations are available at \oururl. 
\end{proof}

\begin{proof}[Proof of Theorem~\ref{theo: krivextension}]
Let $0.579\leq \alpha<0.6$ and $\alpha_i:=0.579+0.0005i$, $\beta_i:=(2\alpha-1)/4-\frac{1}{4000}$ for $i\in \{0,1,\ldots,42\}$. There exists $j\in\{0,1,\ldots,41\}$ such that $\alpha_j\leq \alpha\leq \alpha_{j+1}$. Let $G$ be a triangle-free $n$-vertex graph. If $\alpha(G)\geq 1-\alpha$ then, by Lemma~\ref{lem: indseterdos}, there exists a $\lfloor \alpha n\rfloor$-subset of $V(G)$ spanning at most $\beta n^2$ edges, where $\beta=(2\alpha-1)/4$. If $\alpha(G)\leq 1-\alpha\leq 1-\alpha_j$, then by Lemma~\ref{lem: krivextendisc} there exists $A\subset V(G)$ of size $|A|=\lfloor \alpha_j n\rfloor$ spanning at most $\beta_j n^2$ edges. An arbitrary set $B\supseteq A$ of size $\lfloor \alpha n\rfloor$ spans at most
\begin{align*}
e(B)&=
    e(A)+e(B\setminus A)+e(A,B)\leq \beta_j n^2+\sum_{x\in B\setminus A}\deg(x) \leq \beta_j n^2+ \left(\lfloor \alpha n\rfloor-\lfloor \alpha_j n\rfloor\right)(1-\alpha)n \\
    &\leq \beta_j n^2+\frac{n^2}{4000}\leq \frac{2\alpha-1}{4}n^2.
\end{align*}
\end{proof}

\section{Balanced 3-partition: proof of Theorem~\ref{theo:baltri}}
\label{sec:baltri}
We start proving Theorem~\ref{theo:baltri} for graphs with large independence number.
\begin{lemma}
\label{lem:triplargeind}
There exists $n_0$ such that for every $n\geq n_0$ divisible by three, the following holds. Let $G$ be a triangle-free graph with $\alpha(G)\geq n/3$. Then 
\begin{align*}
    D_3^b(G)\leq \frac{n^2}{36}.
\end{align*}
Further, if $D_3^b(G)=\frac{n^2}{36}$, then $G\cong K_{\frac{5n}{6},\frac{n}{6}}$ or $G\cong K_{\frac{n}{2},\frac{n}{2}}$. 
\end{lemma}
\begin{proof}
Let $I$ be an independent set of size $n/3$. By applying Theorem~\ref{theo:balancedcut} to $G[V\setminus I]$, there exists a balanced 2-partition $V\setminus I=A\cup B$ satisfying
\begin{align*}
     e(A)+e(B)\leq \frac{\left(\frac{2n}{3}\right)^2}{16}\leq \frac{n^2}{36}.
\end{align*}
Now, $V(G)=I\cup A\cup B$ is a balanced partition with at most $n^2/36$ class-edges, proving the first part of this lemma. For the second part, assume that $G$ is an $n$-vertex triangle-free graph satisfying $D_3^b(G)=n^2/36$ and $\alpha(G)\geq n/3$. If 6 does not divide $n$, then $n^2/36$ is not an integer. Therefore, $n$ is divisible by 6. Let $V_1$ be an independent set of size $n/3$. Now 
\begin{align*}
    D_2^b(G[V\setminus V_1])=e(V_1)+ D_2^b(G[V\setminus V_1]) \geq D_3^b(G)\geq \frac{n^2}{36},
\end{align*}
and therefore by applying Theorem~\ref{theo:balancedcut} to $G[V\setminus V_1]$, we have $G[V\setminus V_1]\cong K_{\frac{n}{6},\frac{n}{2}}$. Hence, there exists a partition, $V= V_1 \cup V_2\cup V_3$ such that $|V_1|=n/3$, $|V_2|=n/6$, $|V_3|=n/2$, the sets $V_1,V_2$ and $V_3$ are independent and $V_2\cup V_3$ spans a complete bipartite graph. We split up $V_3=V_3' \cup V_3''$ into arbitrary sets of size $|V_3'|=n/3$ and $|V_3''|=n/6$. Now 
\begin{align*}
    D_2^b(G[V\setminus V_3'])=e(V_3')+ D_2^b(G[V\setminus V_3']) \geq D_3^b(G)\geq \frac{n^2}{36},
\end{align*}
and therefore by applying Theorem~\ref{theo:balancedcut} to $G[V\setminus V_3']$, we have $G[V\setminus V_3']\cong K_{\frac{n}{6},\frac{n}{2}}$. Let $A\cup B$ with $|A|=n/6$ and $|B|=n/2$ be the classes of the induced copy of $K_{\frac{n}{6},\frac{n}{2}}$ in $G[V\setminus V_3']$.  Since $V_2\cup V_3''$ spans a complete bipartite graph we have either $A=V_3''$ and $B=V_1\cup V_2$ or $A=V_2$ and $B=V_1\cup V_3''$.

\underline{Case 1: $A=V_3''$ and $B=V_1\cup V_2$:}\\
We will observe that $G\cong K_{\frac{n}{2},\frac{n}{2}}$ with classes $U_1=V_3$ and $U_2=V_1\cup V_2$. The set $U_2$ is independent because every vertex from $V_3''$ is adjacent to every vertex from $U_2$ and $G$ is triangle-free. Therefore, $G$ is a bipartite graph with classes $U_1$ and $U_2$. Towards contradiction, assume that there exists a pair $u_1u_2\notin E(G)$ with $u_1\in U_1$ and $u_2\in U_2$. Now, take a balanced 3-partition $W_1\cup W_2 \cup W_3$ with $W_1\subset U_1\setminus \{u_1\}$ and $W_2\subset U_2\setminus \{u_2\}$. We have $e(W_1)=e(W_2)=0$ and $e(W_3)<n^2/36$, because $u_1u_2\notin E(G)$. Hence,
\begin{align*}
    \frac{n^2}{36}> e(W_1)+e(W_2)+e(W_3)\geq D_b^3(G)\geq \frac{n^2}{36},
\end{align*}
a contradiction. We conclude that $G\cong K_{\frac{n}{2},\frac{n}{2}}$. 

\underline{Case 2: $A=V_2$ and $B=V_1\cup V_3''$:}\\
In this case $G\cong K_{\frac{5n}{6},\frac{n}{6}}$ with classes $U_1= V_1 \cup V_3' \cup V_3''$ and $U_2= V_2$ because every vertex from $V_2$ is adjacent to every vertex from $V_1 \cup V_3' \cup V_3''$ and $G$ is triangle-free. \end{proof}
Next, we prove Theorem~\ref{theo:baltri} for graphs with small independence number.
\begin{lemma}
\label{lem:tripsmallind}
There exists $n_0$ such that for every $n\geq n_0$ divisible by three, the following holds. Let $G$ be a triangle-free graph with $\alpha(G)\leq n/3$. Then there exists $\varepsilon>0$ such that
\begin{align*}
    D_3^b(G)< \left(\frac{1}{36}-\varepsilon\right) n^2.
\end{align*}
\end{lemma}
\begin{proof}
We prove this lemma with $\varepsilon=10^{-4}$. Assume that $G$ is a triangle-free $n$-vertex graph with $D_3^b(G)\geq \left(\frac{1}{36}-\varepsilon \right) n^2$ and $\alpha(G)\leq n/3$. Since $G$ is triangle-free, we can assume $\Delta(G)\leq \alpha(G)\leq n/3$. Hence, 
\begin{align}
\label{form2: max degree1over3}
\Flu2 \leq \frac{1}{3}.
\end{align}
We define balanced 3-partitions depending on the adjacencies of an arbitrary fixed vertex $x\in V$. Denote by $\mathcal{P}$ the set of balanced $3$-partitions $V=A_1\cup A_2\cup A_3$ such that $N(x)\subseteq A_1$.
Now, by averaging over all choices of $\mathcal{P}$, we have
\begin{align}
\label{ineqflag31bal1}
    \frac{1}{\binom{n-1}{2}|\mathcal{P}|}\sum_{(A_1,A_2,A_3)\in\mathcal{P}}e(A_1)=\Fluu212 \frac{\left( \frac{1}{3} - \Flu2  \right)}{\left( \Flu1 \right)} + \Fluu112 \left(\frac{  \frac{1}{3} - \Flu2 }{ \Flu1 }\right)^2
\end{align}
and for each of $\ell\in\{2,3\}$ we have
\begin{align}
\label{ineqflag31bal2}
        \frac{1}{\binom{n-1}{2}|\mathcal{P}|}\sum_{(A_1,A_2,A_3)\in\mathcal{P}}e(A_\ell)= \Fluu112  \left(\frac{\frac{1}{3}}{ \Flu1 }\right)^2.
\end{align}
For every balanced 3-partition $V(G)=A_1\cup A_2\cup A_3$ we have $(1/36-\varepsilon) n^2 \leq D_3^b(G)\leq e(A_1)+e(A_2)+e(A_3)$ by assumption. Combining this with \eqref{ineqflag31bal1} and \eqref{ineqflag31bal2}, we get
\[
\label{cut1}
\frac{1}{18}-2\varepsilon \leq \Fluu212 \frac{\left( \frac{1}{3} - \Flu2  \right)}{\left( \Flu1 \right)} + \Fluu112 \left(\frac{ 2\left(\frac{1}{3}\right)^2+\left( \frac{1}{3} - \Flu2  \right)^2}{\left( \Flu1 \right)^2}\right).
\]
This can be rewritten as 
\begin{align}
0&\leq \Fluu212 \left(1 - 3\Flu2  \right)\Flu1 +  \Fluu112 \left( 1-2\Flu2 + 3\Fluu221\right)
-
\left(\frac{1}{18}-2\varepsilon\right)3\left( \Flu1\right)^2.
\label{vertexcutrandom}
\end{align}
Other balanced $3$-partitions are defined based on the adjacencies of an edge $uv\in E(G)$. Denote by $\mathcal{P}'$ the set of balanced $3$-partitions $V(G)=A_1\cup A_2 \cup A_3$ such that $N(u)\subseteq A_1$ and $N(v)\subseteq A_2$. By averaging over all choices of $\mathcal{P}'$, we get
\begin{align}
\label{ineqflag3bal1}
    \frac{1}{\binom{n-2}{2}|\mathcal{P}'|}\sum_{(A_1,A_2,A_3)\in\mathcal{P}'}e(A_1)=\Flluu212112 \frac{\left( \frac{1}{3} - \Fllu221  \right)}{ \Fllu211 } + \Flluu211112 \left(\frac{  \frac{1}{3} - \Fllu221  }{ \Fllu211   }\right)^2,
\end{align}
\begin{align}
\label{ineqflag3bal2}
    \frac{1}{\binom{n-2}{2}|\mathcal{P}'|}\sum_{(A_1,A_2,A_3)\in\mathcal{P}'}e(A_2)=\Flluu211212 \frac{\left( \frac{1}{3} - \Fllu212  \right)}{ \Fllu211  } + \Flluu211112 \left(\frac{  \frac{1}{3} - \Fllu212   }{ \Fllu211   }\right)^2
\end{align}
and
\begin{align}
\label{ineqflag3bal3}
    \frac{1}{\binom{n-2}{2}|\mathcal{P}'|}\sum_{(A_1,A_2,A_3)\in\mathcal{P}'}e(A_3)= \Flluu211112 \left(\frac{  \frac{1}{3}    }{ \Fllu211   }\right)^2.
\end{align}
For every balanced 3-partition $V(G)=A_1\cup A_2\cup A_3$ we have $(1/36-\varepsilon) n^2 \leq D_3^b(G)\leq e(A_1)+e(A_2)+e(A_3)$ by assumption. Combining this with \eqref{ineqflag3bal1}, \eqref{ineqflag3bal2} and \eqref{ineqflag3bal3} we get
\begin{align*}
\frac{1}{18}-2\varepsilon&\leq \Flluu212112 \frac{\left( \frac{1}{3} - \Fllu221  \right)}{\Fllu211} \\
&+ \Flluu211212 \frac{\left( \frac{1}{3} - \Fllu212  \right)}{\Fllu211} + \Flluu211112 \frac{ \frac{1}{9} + \left(\frac{1}{3} - \Fllu212\right)^2 +\left(\frac{1}{3} - \Fllu221\right)^2  }{  \left(\Fllu211  \right)^2 }.
\end{align*}
This can be rewritten as
\label{edgecutrandom}
\begin{align}
\nonumber
&0\leq -\left(\frac{1}{18}-2\varepsilon \right) \left(3\Fllu211\right)^2 + \Flluu212112 \left( 3 - 9\Fllu221  \right) \Fllu211  \\
&+ \Flluu211212 \left(3- 9\Fllu212  \right)  \Fllu211  
\label{edgecutrandom}
+ \Flluu211112 \left(1 + \left(1 - 3\Fllu212\right)^2 +\left(1 - 3\Fllu221\right)^2 \right).
\end{align}
We define other balanced $3$-partitions based on the adjacencies of an edge $uv\in E(G)$ and an independent vertex $w$. Denote by $\mathcal{P}'$ the set of balanced $3$-partitions $V(G)=A_1\cup A_2 \cup A_3$ such that $N(u)\subseteq A_1$,  
$N(v)\setminus N(w) \subseteq A_2$, and $N(w)\setminus N(u) \subseteq A_3$.
By averaging over all choices of $\mathcal{P}'$, we get $\frac{1}{\binom{n-3}{2}|\mathcal{P}'|}\sum_{(A_1,A_2,A_3)\in\mathcal{P}'}e(A_1)=$
\begin{align}
\label{ineqflag3balX1}
    \left(\Fllluu1211112+\Fllluu1211122\right) \frac{\left( \frac{1}{3} - \Flllu212111 - \Flllu212112  \right)}{ \Flllu211111} + \Fllluu1111112 \left(\frac{  \frac{1}{3} - \Flllu212111 - \Flllu212112  }{ \Flllu211111   }\right)^2.
\end{align}
Similarly we get $\frac{1}{\binom{n-3}{2}|\mathcal{P}'|}\sum_{(A_1,A_2,A_3)\in\mathcal{P}'}e(A_2)=$
\begin{align}
\label{ineqflag3balX2}
    \left(\Fllluu1112112\right) \frac{\left( \frac{1}{3} - \Flllu211121  \right)}{ \Flllu211111} + \Fllluu1111112 \left(\frac{  \frac{1}{3}  - \Flllu211121}{ \Flllu211111   }\right)^2.
\end{align}
Finally we get $\frac{1}{\binom{n-3}{2}|\mathcal{P}'|}\sum_{(A_1,A_2,A_3)\in\mathcal{P}'}e(A_3)=$
\begin{align}
\label{ineqflag3balX3}
    \left( \Fllluu1111212+\Fllluu1112122 \right) \frac{\left( \frac{1}{3} - \Flllu211112  - \Flllu211122 \right)}{ \Flllu211111} + \Fllluu1111112 \left(\frac{  \frac{1}{3}  - \Flllu211112  - \Flllu211122 }{ \Flllu211111   }\right)^2.
\end{align}
Multiplying by the denominator and summing up \eqref{ineqflag3balX1}, \eqref{ineqflag3balX2}, and \eqref{ineqflag3balX3}, and using $(1/36-\varepsilon) n^2 \leq D_3^b(G)\leq e(A_1)+e(A_2)+e(A_3)$, we get
\begin{align}
\nonumber
&0\leq -\left(\frac{1}{18}-2\varepsilon \right)\left( \Flllu211111 \right)^2\\
\nonumber
&+\left(\Fllluu1211112+\Fllluu1211122\right) \left( \frac{1}{3} - \Flllu212111 - \Flllu212112  \right) \Flllu211111 + \Fllluu1111112 \left(\frac{1}{3} - \Flllu212111 - \Flllu212112\right)^2
\label{eq:falabel3part}
\\
&+\left(\Fllluu1112112\right) \left( \frac{1}{3} - \Flllu211121   \right) \Flllu211111 + \Fllluu1111112 \left(\frac{1}{3} - \Flllu211121 \right)^2\\
\nonumber
&+
\left( \Fllluu1111212+\Fllluu1112122 \right) \left( \frac{1}{3} - \Flllu211112  - \Flllu211122 \right) \Flllu211111 + \Fllluu1111112 \left(\frac{1}{3}  - \Flllu211112 - \Flllu211122\right)^2.
\end{align}

For $\varepsilon = 0.0001$, flag algebras calculation can conclude there is no sufficiently large graph satisfying all  \eqref{form2: max degree1over3}, \eqref{edgecutrandom}, and \eqref{eq:falabel3part} at the same time.
This part of the proof is computer assisted. The files needed to perform the corresponding calculations are available at \oururl. 
\end{proof}
Lemmas~\ref{lem:triplargeind} and \ref{lem:tripsmallind} together imply Theorem~\ref{theo:baltri}.

\section{Unbalanced 2-partition: proof of Theorem~\ref{unbaltheo}}
\label{sec:theounb}
Let $\alpha$ fixed, $\frac{1}{2}\leq \alpha\leq 1$. For an $n$-vertex graph $G$, we denote by $D_\alpha(G)$ the minimum of $e(G[A])+e(G[A^c])$ over all vertex subsets $A$ of size $\lfloor\alpha n\rfloor$. 
\begin{lemma}
\label{unbalrandomcut}
Let $\alpha> 1/2$. If
\begin{align*}
    e(G)\leq\frac{-\alpha^2+4\alpha-2}{4(2\alpha-1)}n^2 \quad \quad \text{ then} \quad \quad
    D_\alpha(G)\leq \frac{2\alpha- 1}{4}n^2+o(n^2).
\end{align*}
\end{lemma}
\begin{proof}
Choose a random set $A$, where each vertex is included in $A$ with probability $2\alpha-1$ independently of each other. Then, with high probability,
\begin{gather}
\nonumber
    |A|=(2\alpha-1)n +o(n), \\
\label{probrandomAAc1}
e(A)=(2\alpha-1)^2e(G)+o(n^2) \quad \text{and} \quad e(A,A^c)=2(2\alpha-1)(2-2\alpha)e(G)+o(n^2).
\end{gather}
In particular, there exists $A\subseteq V$ of size exactly $2\lfloor\alpha n\rfloor-n$ satisfying \eqref{probrandomAAc1}. By Theorem~\ref{theo:balancedcut} there exists a balanced 2-partition $A^c=A_1\cup A_2$ such that 
\begin{align*}
    e(A_1)+e(A_2)\leq \frac{(2(1-\alpha))^2}{16}n^2 +o(n^2)=\frac{(1-\alpha)^2}{4}n^2+o(n^2).
\end{align*}
Now, by considering the $2$-partition $(A\cup A_1)\cup A_2$,
\begin{align}
\label{ineq:cutalphado1}
    D_\alpha(G)\leq e(A\cup A_1)+e(A_2)=e(A)+e(A,A_1)+e(A_1)+e(A_2)
\end{align}
and by considering the $2$-partition $(A\cup A_2)\cup A_1$,
\begin{align}
\label{ineq:cutalphado2}
    D_\alpha(G)\leq e(A\cup A_2)+e(A_1)=e(A)+e(A,A_2)+e(A_2)+e(A_1). 
\end{align}
Adding up \eqref{ineq:cutalphado1} and \eqref{ineq:cutalphado2} we get
\begin{align*}
    2D_\alpha(G)&\leq 2e(A)+e(A,A^c)+2e(A_1)+2e(A_2)\leq 2(2\alpha-1)^2e(G)+2(2\alpha-1)(2-2\alpha)e(G)\\ &+\frac{(1-\alpha)^2}{2}n^2+o(n^2)= \left(4\alpha-2\right)e(G)+ \frac{(1-\alpha)^2}{2}n^2+o(n^2).
\end{align*}
Thus, by dividing by 2 and using $e(G)\leq \frac{-\alpha^2+4\alpha-2}{4(2\alpha-1)}n^2$, we get
\begin{align*}
    D_\alpha(G)&\leq \left(2\alpha-1\right)e(G)+ \frac{(1-\alpha)^2}{4}n^2+o(n^2)\leq  \frac{2\alpha-1}{4}n^2+o(n^2).
\end{align*}
\end{proof}

\begin{lemma}
\label{unbalanindep}
Let $\alpha\geq 0.717$ and $G$ be an $n$-vertex triangle-free graph. If $\alpha(G)\geq n/2$, then 
$D_\alpha(G)\leq \frac{2\alpha-1}{4}n^2+o(n^2)$.
\end{lemma}
\begin{proof}
Let $I$ be an independent set of size $|I|=\lceil n/2\rceil$. Assume that there exists an independent set $I'\subset V\setminus I$ of size $|I'|=n-\lfloor\alpha n \rfloor$. Take a largest independent set $I''\subseteq V\setminus I'$ and note that $|I''|\geq|I|\geq n/2$. Now,
\begin{align*}
    e(V\setminus I')\leq \sum_{v\in V\setminus (I'\cup I'')}|N(v)\cap(V\setminus I')|\leq (\lfloor\alpha n\rfloor -|I''|)|I''|\leq \frac{2\alpha-1}{4}n^2,
\end{align*}
implying
\begin{align*}
    D_\alpha(G)\leq e(I')+e(V\setminus I')\leq\frac{2\alpha-1}{4}n^2.
\end{align*}
Thus, we can assume that there does not exist an independent set $I'\subset V\setminus I$ of size $|I'|=n-\lfloor \alpha n\rfloor$. Let $B\subset V\setminus I$ be a largest independent set in $V\setminus I$. Thus $|B|<n-\lfloor \alpha n\rfloor$. 
Since $G$ is triangle-free and by the maximality of $B$, we have $\deg(v)\leq |B|$ for every $v\in I$. Furthermore, for the same reason, for every $v\in V\setminus (I\cup B)$, we have $|N(v)\setminus I|\leq |B|$. Therefore, \begin{align*}
    e(G)&=e(I,V\setminus I)+e(V\setminus I)\leq\sum_{v\in I}\deg(v)+\sum_{v\in V\setminus(I\cup B)}|N(v)\setminus I|\leq |I||B|+(n-|I|-|B|)|B|\\
    &=(n-|B|)|B|
    \leq \alpha (1-\alpha)n^2+O(n)\leq \frac{-\alpha^2+4\alpha-2}{4(2\alpha-1)}n^2,
\end{align*}
where the last inequality holds for $n$ sufficiently large and $\alpha\geq 0.717$. By Lemma~\ref{unbalrandomcut}, we have $D_\alpha(G)\leq \frac{2\alpha- 1}{4}n^2+o(n^2)$, completing the proof of this lemma.
\end{proof}

\begin{lemma}
\label{unbal2indep}
Let $\alpha\geq 0.717$ and $G$ be an $n$-vertex triangle-free graph. If there exists two disjoint independent sets each of size $n-\alpha n$, then  $D_\alpha(G)\leq \frac{2\alpha-1}{4}n^2+o(n^2)$.
\end{lemma}
\begin{proof}
Towards contradiction, assume that $G$ is an $n$-vertex triangle-free graph satisfying $D_\alpha(G)>\frac{2\alpha-1}{4}n^2$ and containing two disjoint independent sets $I_1,I_2$ each of size $n-\lfloor\alpha n\rfloor$. Denote $A:=V\setminus (I_1 \cup I_2)$.  Then,
\begin{align*}
    e(I_1,A)+e(A)\geq D_\alpha(G)\geq  \frac{2\alpha-1}{4}n^2 \quad \text{and} \quad     e(I_2,A)+e(A)\geq D_\alpha(G) \geq \frac{2\alpha-1}{4}n^2. 
\end{align*}
Summing up those two inequalities, we have
\begin{align}
\label{indepdegsum}
    \sum_{v\in A}\deg(v)=2e(A)+e(I_1,A)+e(I_2,A)\geq \frac{2\alpha-1}{2}n^2.
\end{align}
On the other side, 
\begin{align*}
    \sum_{v\in A}\deg(v)\leq \alpha(G)|A|=\alpha(G)(2\alpha-1)n.
\end{align*}
Thus, $\alpha(G)\geq n/2$. By Lemma~\ref{unbalanindep} we have $D_\alpha(G)\leq \frac{2\alpha-1}{4}n^2+o(n^2)$.

\end{proof}

\begin{proof}[Proof of Theorem~\ref{unbaltheo}]
By Lemmas~\ref{unbalanindep} and \ref{unbal2indep} we can assume that there does not exist an independent set of size $n/2$ and also there do not exist two disjoint independent sets of size $n-\lfloor\alpha n\rfloor$ each. Let $I$ be a largest independent set in $G$. If $|I|\leq n-\lfloor\alpha n\rfloor$, then
\begin{align*}
    e(G)=\frac{1}{2}\sum_{v\in V(G)}\deg(v)\leq \frac{n}{2}(n-\lfloor\alpha n\rfloor)\leq \frac{-\alpha^2+4\alpha-2}{4(2\alpha-1)}n^2,
\end{align*}
where the last inequality holds for $\alpha> \frac{2}{3}$ and $n$ sufficiently large. By Lemma~\ref{unbalrandomcut} we have $D_\alpha(G)\leq \frac{2\alpha-1}{4}n^2+o(n^2)$. Thus, we can assume that $n-\lfloor\alpha n\rfloor\leq|I|\leq n/2$. Let $B\subseteq V\setminus I$ be a largest independent set in $V\setminus I$. We have $|B|< n-\lfloor\alpha n\rfloor$, because otherwise there were two disjoint independent sets each of size $n-\lfloor \alpha n\rfloor$. Since $G$ is triangle-free and by the maximality of $B$, we have $\deg(v)\leq |B|$ for every $v\in I$. Furthermore, for the same reason, for every $v\in V\setminus I$, we have $|N(v)\cap (V\setminus I)|\leq |B|$. Therefore, 
\begin{align*}
    e(G)&=e(I,V\setminus I) +e(V\setminus I)\leq\sum_{v\in I}\deg(v)+\sum_{v\in V\setminus(I\cup B)}|N(v)\cap B|\\
    &\leq |I||B|+(n-|I|-|B|)|B|=(n-|B|)|B|
    \leq \alpha (1-\alpha)n^2+o(n^2)\leq \frac{-\alpha^2+4\alpha-2}{4(2\alpha-1)}n^2,
\end{align*}
where the last inequality holds for $n$ sufficiently large and $\alpha\geq 0.717$.
By Lemma~\ref{unbalrandomcut}, we have $D_\alpha(G)\leq \frac{2\alpha-1}{4}n^2+o(n^2)$.
\end{proof}

\section{Clique-free graphs: Proof of Theorem~\ref{KeevashSudakovextension}}
\label{sec:clique-freegraphs}

We will make use of the following two lemmas proved by Keevash and Sudakov~\cite{MR1967401}.
\begin{lemma}[Keevash, Sudakov~\cite{MR1967401}]
\label{lem: edgesinpartitiongraph}
Let $r\geq2$ and let $G$ be a $K_{r+1}$-free graph where the vertex set of $G$ is a union of three disjoint sets $X$, $Y$ and $Z$ such that $X$ is an independent set and $Y$ can be covered by a collection of vertex disjoint copies of $K_r$'s and $|X| = xn, |Y | = yn$ and $|Z| = zn$. Then
\begin{align*}
    \frac{1}{n^2}e(H)\leq \frac{r-1}{2r}(x+y+z)^2-\frac{((r-1)x-z)^2}{2r(r-1)}+\frac{1}{n^2}e(Z)-\frac{r-2}{2(r-1)}z^2.
\end{align*}
\end{lemma}

\begin{lemma}[Keevash, Sudakov~\cite{MR1967401}]
\label{lem:KeesudindesetKr}
Let $r\geq2$ be an integer and $G$ a $K_{r+1}$-free graph with $n$ vertices and $m$ edges. Then $G$ contains an independent set of size at least $2(r-1)\frac{m}{n}-(r-2)n.$
\end{lemma}
The following is a version of Corollary~3.4 from \cite{MR1967401} in our setting.
\begin{corl}
\label{corlindepsetKe}
Let $r\geq2$ be an integer and $k\geq1$. Then there exists $c=c(k,r)>0$ such that the following holds. Fix an arbitrary $\alpha$, with $c\leq\alpha\leq1$. Let $G$ be a $K_{r+1}$-free graph with $n$ vertices such that for every set $A$ of exactly $\lfloor \alpha n\rfloor$ vertices of $G$ we have $e(A)+e(A^c)\geq \frac{r-1}{2r}(2\alpha-1)n^2$ edges. Then $G$ contains an independent set of size at least $k(1-\alpha)n+o(n)$.
\end{corl}
\begin{proof}
Denote $m$ the number of edges of $G$ and let $c=c(k,r)$ be sufficiently large. Take a vertex subset $A$ of size $\lfloor\alpha n\rfloor$ uniformly at random.  Now, for any edge $e\in E(G)$ we have
\begin{align*}
    \mathbb{P}(e\in E(G[A]) \cup E(G[A^c]))&= \frac{\binom{n-2}{\lfloor \alpha n\rfloor-2}}{\binom{n}{\lfloor \alpha n\rfloor}} + \frac{\binom{n-2}{\lfloor\alpha n\rfloor}}{\binom{n}{\lfloor \alpha n\rfloor}}\\
    &=\frac{\lfloor \alpha n\rfloor (\lfloor \alpha n\rfloor-1)}{n(n-1)}+ \frac{(n-\lfloor \alpha n \rfloor)(n-\lfloor \alpha n \rfloor-1)}{n(n-1)}\leq \alpha^2+(1-\alpha)^2+o(1).
\end{align*}
Thus there exists a vertex subset $A$ of size exactly $\lfloor\alpha n\rfloor$ such that 
\begin{align*}
    e(A)+e(A^c)\leq m(\alpha^2+(1-\alpha^2)+o(1)),
\end{align*}
implying that
\begin{align*}
    m\geq \frac{e(A)+e(A^c)}{\alpha^2+(1-\alpha^2)}+o(n^2)\geq  \frac{r-1}{2r}\frac{2\alpha-1}{\alpha^2+(1-\alpha)^2}n^2+o(n^2).
\end{align*}
By Lemma~\ref{lem:KeesudindesetKr} there exists an independent set of size at least
\begin{align*}
     \left(\frac{(r-1)^2}{r} \frac{2\alpha-1}{\alpha^2+(1-\alpha)^2}-(r-2)\right)n+o(n).
\end{align*}
Therefore, it suffices to show that
\begin{align}
\label{inequalityalphar}
        \frac{(r-1)^2}{r} \frac{2\alpha-1}{\alpha^2+(1-\alpha)^2}-(r-2)\geq k(1-\alpha).
\end{align}
Since we chose $c$ sufficiently large, we have for every $\alpha\geq c$
\begin{align*}
    \frac{2\alpha-1}{\alpha^2+(1-\alpha)^2} \geq \frac{(r-2)r+\frac{1}{2}}{(r-1)^2},
\end{align*}
because $f(\alpha):= \frac{2\alpha-1}{\alpha^2+(1-\alpha)^2}$ is a non-decreasing continuous function for $0\leq\alpha\leq1$ and $f(1)=1$. To see this, note that $f'(\alpha)=\frac{4\alpha(1-\alpha)}{(2\alpha^2-2\alpha+1)^2}\geq 0$. We conclude
\begin{align*}
    \frac{(r-1)^2}{r} \frac{2\alpha-1}{\alpha^2+(1-\alpha)^2}-(r-2)\geq \frac{(r-2)r+\frac{1}{2}}{r}-(r-2)=\frac{1}{2r}\geq k(1-\alpha),
\end{align*}
proving \eqref{inequalityalphar}.
\end{proof}

First we prove Theorem~\ref{KeevashSudakovextension} asymptotically and then we will argue that this already implies the exact version.
\begin{prop}
\label{KeevashSudakovextensionprop}
Let $r\geq 1$. There exists $c_r<1$ such that the following holds for every $\alpha$ such that $c_r\leq \alpha\leq 1$. Let $G$ be a $K_{r+1}$-free graph on $n$ vertices, then $G$ contains a set of $\left\lfloor \alpha n\right\rfloor$ vertices $A$ such that 
\begin{align*}
    e(A)+e(A^c)\leq  \frac{r-1}{2r}(2\alpha-1)n^2+o(n^2).
\end{align*}
\end{prop}
\begin{proof}[Proof of Proposition~\ref{KeevashSudakovextensionprop}]
Throughout this proof we omit all floor signs and lower-order error terms for readability. Since all rounding errors change the number of edges only by $o(n^2)$, we can safely ignore them.

We use induction on $r$. For $r=1$ the statement is trivial. For $r=2$, Proposition~\ref{KeevashSudakovextensionprop} follows from Theorem~\ref{unbaltheo} with $c_2=0.74$. Now, assume $r\geq 3$. Choose a constant $k_r>\sqrt{2}$ such that 
\begin{align*}
    1-\frac{1}{(r-1)(k_r-\sqrt{2})}\geq c_{r-1}
\end{align*}
and choose $c_r=c_r(k_r,r)<1$ sufficiently large. Towards contradiction, we assume that there exists an $n$-vertex $K_{r+1}$-free graph $G$ such that for some $\alpha\in [c_r,1)$ and every set $A$ of $\alpha n$ vertices we have 
\begin{align}
\label{cutassumption}
    e(A)+e(A^c)> \frac{r-1}{2r}(2\alpha-1)n^2.
\end{align}

By Corollary~\ref{corlindepsetKe}, the graph $G$ contains an independent set $U$ of size $k_r(1-\alpha)n$. Let $T$ be the largest subset of $V(G)\setminus U$ which can be covered by vertex-disjoint cliques of size $r$. Let $S=V(G)\setminus(U\cup T)$ and set $t=|T|/n$ and $s =|S|/n$. The subgraph $G[S]$ is $K_r$-free and $k_r(1-\alpha) + t + s = 1$.

Let $X_1$ be a subset of $U$ of size $(k_r-1)(1-\alpha)n$, and let $Y_1=T,Z_1=S$. Denote by $H_1$ the subgraph $G[X_1 \cup Y_1 \cup Z_1]$. Applying Lemma~\ref{lem: edgesinpartitiongraph} to $H_1$, we get
\begin{align}
\label{ineq: partlemma1}
\nonumber
    \frac{e(H_1)}{n^2}&\leq \frac{r-1}{2r}((k_r-1)(1-\alpha)+t+s)^2-\frac{((r-1)(k_r-1)(1-\alpha)-s)^2}{2r(r-1)}+\frac{1}{n^2}e(S)-\frac{r-2}{2(r-1)}s^2\\
    &\leq \frac{r-1}{2r}\alpha^2-\frac{((r-1)(k_r-1)(1-\alpha)-s)^2}{2r(r-1)},
\end{align}
where in the last inequality we used $(k_r-1)(1-\alpha)+t+s=\alpha$ and applied Tur\'an's theorem to $G[S]$. Since $|V(H_1)|=\alpha n$ and $U\setminus X_1$ is an independent set, we have by assumption \eqref{cutassumption}
\begin{align}
\label{ineq:applyassumption1}
    e(H_1)=e(H_1)+e(U\setminus X_1)> \frac{r-1}{2r}(2\alpha-1)n^2.
\end{align}
Combining \eqref{ineq: partlemma1} with \eqref{ineq:applyassumption1}, we get
\begin{align*}
   \frac{((r-1)(k_r-1)(1-\alpha)-s)^2}{2r(r-1)}< \frac{r-1}{2r}(1-\alpha)^2,
\end{align*}
implying
\begin{align*}
    \lvert(r-1)(k_r-1)(1-\alpha)-s\rvert< (r-1)(1-\alpha).
\end{align*}
Thus, we have $s< (r-1)k_r(1-\alpha)$. Set 
\begin{align*}
 q =
\begin{cases}
\frac{(1-\alpha)-t}{r} & \quad \text{if } t < (1-\alpha) \\
0 & \quad \text{otherwise}.
\end{cases}
\end{align*}
Note that $s\geq rq$, because otherwise $U$ would be an independent set of size at least $\alpha n$. This is not possible, because an independent set  $I\subseteq U$ of size exactly $\alpha n$ satisfies
\begin{align*}
    e(I)+e(I^c)=e(I^c)\leq \frac{r-1}{2r}(1-\alpha)^2\leq \frac{r-1}{2r}(2\alpha-1)
\end{align*}
for $\alpha\geq c_r$, a contradiction.  

Let $X_2$ be a subset of $U$ of size $(k_r(1-\alpha)-q)n$, let $Z_2$ be a subset of $S$ of size $(s-(r-1)q)n$ and let $Y_2$ be a subset obtained by deleting $(\frac{1-\alpha}{r}-q)n$
disjoint copies of $K_r$ from the set $T$. Note that
\begin{align*}
    |X_2| + |Y_2| + |Z_2|&=(k_r(1-\alpha)-q)n+(1-\alpha-qr)n+(s-(r-1)q)n\\&=(k_r-1)(1-\alpha)+sn+tn=\alpha n.
\end{align*} Applying Lemma~\ref{lem: edgesinpartitiongraph} to $H_2:=G[X_2\cup Y_2 \cup Z_2]$, we get
\begin{align}
\label{ineq: partlemma2}
\nonumber
    \frac{e(H_2)}{n^2}&\leq \frac{r-1}{2r}\alpha^2-\frac{((r-1)(k_r(1-\alpha)-q)-(s-(r-1)q))^2}{2r(r-1)}\\
    &+\frac{1}{n^2}e(Z_2)-\frac{r-2}{2(r-1)}|Z_2|^2= \frac{r-1}{2r}\alpha^2-\frac{((r-1)k_r(1-\alpha)-s)^2}{2r(r-1)}.
\end{align}
On the other side,  
\begin{align*}
    \frac{r-1}{2r}(2\alpha-1) < \frac{e(H_2)+e(V\setminus(X_2\cup Y_2 \cup Z_2))}{n^2} \leq \frac{e(H_2)}{n^2}+ \frac{r-1}{2r}(1-\alpha)^2,
\end{align*}
we have $\frac{e(H_2)}{n^2}>\frac{r-1}{2r}(2\alpha-1-(1-\alpha)^2)$. Together with \eqref{ineq: partlemma2}, this gives
\begin{align*}
   \frac{((r-1)k_r(1-\alpha)-s)^2}{2r(r-1)}\leq \frac{r-1}{r}(1-\alpha)^2
\end{align*}
implying
\begin{align*}
    \lvert (r-1)k_r(1-\alpha)-s\rvert\leq \sqrt{2}(r-1)(1-\alpha).
\end{align*}
Therefore $s>(r-1)(1-\alpha)(k_r-\sqrt{2})$. Set $\alpha_1=\frac{s-(1-\alpha)}{s}$. Since
\begin{align*}
    \alpha_1=1-\frac{1-\alpha}{s}>1-\frac{1}{(r-1)(k_r-\sqrt{2})},
\end{align*}
the induction assumption holds, and $G[S]$ contains a subset $Z_3$ of size $(s-(1-\alpha))n$ such that
\begin{align}
\label{ineq:edgeZ3+comp}
    e(Z_3)+e(S\setminus Z_3)\leq \frac{r-2}{2(r-1)}(2\alpha_1-1)s^2=\frac{r-2}{2(r-1)}((s-(1-\alpha))^2-(1-\alpha)^2).
\end{align}
Now, let $X_3=U$ and $Y_3=T$. Again, by applying Lemma~\ref{lem: edgesinpartitiongraph} to $H_3=G[X_3\cup Y_3\cup Z_3]$, we get
\begin{align}
\label{ineq: partlemma3}
    \frac{e(H_3)}{n^2}&\leq \frac{r-1}{2r}\alpha^2-\frac{((r-1)k_r(1-\alpha)-(s-(1-\alpha)))^2}{2r(r-1)}+\frac{1}{n^2}e(Z_3)-\frac{r-2}{2(r-1)}(s-(1-\alpha))^2.
\end{align}
By combining \eqref{ineq: partlemma2} with  \eqref{ineq:edgeZ3+comp} we get
\begin{align}
\label{ineq:H3complement}
    \frac{e(H_3)+e(S\setminus Z_3)}{n^2} &\leq \frac{r-1}{2r}\alpha^2-\frac{(((r-1)k_r+1)(1-\alpha)-s)^2}{2r(r-1)}+\frac{r-2}{2(r-1)}(1-\alpha)^2.
\end{align}
On the other side, we have
\begin{align*}
    \frac{e(H_3)+e(S\setminus Z_3)}{n^2} > \frac{r-1}{2r}(2\alpha-1).
\end{align*}
Combining \eqref{ineq: partlemma3} with \eqref{ineq:H3complement} we obtain
\begin{align*}
\frac{(((r-1)k_r+1)(1-\alpha)-s)^2}{2r(r-1)} < \frac{r-1}{2r}(1-\alpha)^2+\frac{r-2}{2(r-1)}(1-\alpha)^2,
\end{align*}
implying
\begin{align*}
\frac{(((r-1)k_r+1)(1-\alpha)-s)^2}{2r(r-1)} < \frac{(1-\alpha)^2}{2r(r-1)}.
\end{align*}
Thus $\lvert ((r-1)k_r+1)(1-\alpha)-s \rvert < (1-\alpha)$ and therefore $s>(r-1)k_r(1-\alpha)$, a contradiction. 
\end{proof}

\begin{proof}[Proof of Theorem~\ref{KeevashSudakovextension}]
Assume, towards contradiction, that there exists a $K_{r+1}$-free graph $H$ with $k$ vertices such that for every vertex subset $A$ of size $\alpha k$ we have
\begin{align*}
    e(A)+e(A^c)\geq  \frac{r-1}{2r}(2\alpha-1) k^2  + 1.
\end{align*}
Let $n$ be an integer sufficiently large and divisible by $k$.
Let $G$ be the graph constructed from $H$ by replacing every vertex $i$ with an independent set $V_i$ of size $n/k$, and for every edge $\{i, j\}\in E(H)$ by a complete graph between the sets $V_i$ and $V_j$. Note that $G$ is a $K_{r+1}$-free $n$-vertex graph. Let $S\subseteq V(G)$ be a set of size $\lfloor \alpha n\rfloor$ realizing the minimum $e(B)+e(B^c)$ among all sets $B\subseteq V(G)$ of size of $\alpha n$. The following argument justifies that we can assume that $S$ either contains or is disjoint from the sets $V_i$'s for all but one index. Towards contradiction and without loss of generality, assume that 
\begin{align*}
    1 \leq|S\cap V_1|\leq |S\cap V_2| <\frac{n}{k}.
\end{align*}
Define
\begin{align*}
w_1&:=S\cap  \{v\in V_\ell :1\ell\in E(H) \text{ for some } \ell\in[k]\setminus\{1,2\} \}, \\
w_1^C&:=S^c\cap  \{v\in V_\ell :1\ell\in E(H) \text{ for some } \ell\in[k]\setminus\{1,2\} \}, \\
w_2&:=S\cap  \{v\in V_\ell :2\ell\in E(H) \text{ for some } \ell\in[k]\setminus\{1,2\} \}, \\
w_2^c&:=S^c\cap \{v\in V_\ell :2\ell\in E(H) \text{ for some } \ell\in[k]\setminus\{1,2\} \}. \\
\end{align*}
Let $i,j\in\{1,2\}$ with $i\neq j$ such that $|w_i|-|w_i^c|\geq |w_j|-|w_j^c|$. If $12\notin E(H)$, then moving vertices in $V_i$ to $S^c$ and vertices in $V_j$ to $S$ one by one does not increase $e(S)+e(S^c)$. If $12\in E(H)$, then moving $\min\{|S\cap V_i|,|S^c\cap V_j| \}$ vertices in $V_i$ to $S^c$ and the same amount of vertices in $V_j$ to $S$ does not increase the number of edges in $e(S)+e(S^c)$. This is because 
\begin{align*}
    |E(G[S\cap (V_1\cup V_2)])\cup E(G[S^c \cap (V_1\cup V_2)])|=|S\cap V_1||S\cap V_2|+|S^c\cap V_1||S^c\cap V_2|,
\end{align*}
which is minimized when $\big\vert|S\cap V_1|-|S\cap V_2|\big\vert$ is as large as possible. We conclude that $S$ either is disjoint from or contains the sets $V_i$'s.

Now, the set $S$ contains $ \alpha k $ sets $V_i$. Denote by $I\subset [k]$ the set of indices $i$ such that $V_i\subseteq S$. We have
\begin{align*}
    e_G(S)+e_G(S^c)=(e_H(I)+e_H(I^c))\frac{n^2}{k^2}\geq \left( \frac{r-1}{2r}(2\alpha-1) k^2 + 1 \right)\frac{n^2}{k^2}=\frac{r-1}{2r}(2\alpha-1) n^2 +\frac{n^2}{k^2},
\end{align*}
contradicting Proposition~\ref{KeevashSudakovextensionprop}.

\end{proof}


\section{Balanced Cuts with bounded class size}
\label{sec:maxkpart}
In this section we prove Theorems~\ref{theo:maxbip} and~\ref{theo:maxtrip}, our two results on balanced partitions with bounded class sizes. The following lemma is a well-known fact which can be proven by a standard probabilistic argument, we omit the proof.
\begin{lemma}
Let $k,n$ be integers where $k$ divides $n$ and $G$ be an $n$-vertex graph. Then there exists a balanced $k$-partition $A_1\cup\ldots\cup A_k$ such that $e(A_i)=e(G)/k^2+o(n^2)$ for every $i\in[k]$.
\label{lemma:maxrandomcut}
\end{lemma}
Further, we will make use of the following exact version of the asymptotic result by Erd\H{o}s, Faudree, Rousseau and Schelp~\cite{MR1273598} on Conjecture~\ref{unbalconj2} for $\alpha=3/4$.

\begin{lemma}
\label{ErFaRoSc}
Let $G$ be an $n$-vertex triangle-free graph, where $n$ is divisible by $4$. Then, there exists a set $A\subseteq V(G)$ of size $3n/4$ such that $e(A)\leq n^2/8$.
\end{lemma}
This lemma follows simply from its asymptotic version by the standard argument presented at the end of Section~\ref{sec:clique-freegraphs}, we omit the details. 
\begin{proof}[Proof of Theorem~\ref{theo:maxbip}]
By Lemma~\ref{lemma:maxrandomcut} there exists a balanced $2$-partition $V(G)=A_1\cup A_2$ such that  $e(A_i)=e(G_n)/4+o(n^2)$ for $i=1,2$. Hence we can assume that $e(G)\ge (2/9-o(1))n^2$.

Since $G$ is triangle-free, for $n$ sufficiently large, there exists an independent set $I$ of size $ n/3$. Note that any set $B$ of size $n/2$ containing $I$ spans at most $n^2/18$ edges. We apply Lemma~\ref{ErFaRoSc} on the complement of $I$: There exists a set $C\subset I^c$ of size $n/2$, spanning at most $n^2/18$ edges.
The complement of $C$ also has this property, as it contains $I$. Thus, $\max\{e(C),e(C^c)\}\leq n^2/18$, completing the proof of Theorem~\ref{theo:maxbip}.
\end{proof}

\begin{proof}[Proof of Theorem~\ref{theo:maxtrip}]
By Lemma~\ref{lemma:maxrandomcut} there exists a balanced $3$-partition $V(G)=A_1\cup A_2\cup A_3$ such that  $e(A_i)=e(G_n)/9+o(n^2)$ for $i=1,2,3$. Hence, we can assume that $e(G_n)\ge (3/16)n^2+o(n^2)$.

First, consider the case $\alpha(G)\geq n/2$. Let $I$ be an independent set of size $ n/2$. We apply the result of Erd\H{o}s, Faudree, Rousseau and Schelp~\cite{MR1273598} on Conjecture~\ref{unbalconj2} on $G[I^c]$ for $\alpha=2/3$: There exists a set $C\subset I^c$ of size $n/3$, spanning at most $n^2/48+o(n^2)$ edges. Take an arbitrary balanced 2-partition $A\cup B$ of $V\setminus C$ such that $|A\cap I|-|B\cap I|\leq1$. Since $A$ contains an independent set of size $n/4+o(n)$, we have 
\begin{align*}
    e(A)\leq \frac{n}{4}\cdot \frac{n}{12}+o(n^2)=\frac{n^2}{48}+o(n^2) \quad \text{and similarly} \quad e(B)\leq \frac{n^2}{48}+o(n^2).
\end{align*}
Thus, $V(G)=A\cup B\cup C$ is a balanced $3$-partition with $\max\{e(A),e(B),e(C)\}\leq n^2/48+o(n^2)$. 

Now, assume that $\alpha(G)<n/2$. By a standard application of the Cauchy-Schwartz inequality, there exists an edge $xy$ satisfying
\begin{align*}
    \deg(x)+\deg(y)\geq \frac{1}{e(G)}\sum_{uv\in E(G)}\deg(u)+\deg(v)= \frac{1}{e(G)}\sum_{v\in V(G)}\deg(v)^2 \geq \frac{4e(G)}{n}\geq \frac{3n}{4}-o(n).
\end{align*} 
Since $G$ is triangle-free, $N(x)$ and $N(y)$ are disjoint independent sets. By $\alpha(G)<n/2$, we get $n/4-o(n)<|N(x)|,|N(y)|< \frac{n}{2}$. In particular, there exists two disjoint independent sets $I_1,I_2$ each of size $n/4+o(n)$. By the result of Erd\H{o}s, Faudree, Rousseau and Schelp~\cite{MR1273598} on Conjecture~\ref{unbalconj2} applied on $G[V\setminus(I_1\cup I_2)]$ for $\alpha=2/3$: There exists a set $C\subset V\setminus(I_1\cup I_2)$ of size $n/3$, spanning at most $n^2/48+o(n^2)$ edges. Take a balanced 3-partition $V(G)=A\cup B\cup C$ such that $I_1\subseteq A, I_2\subseteq B$. This 3-partition has the property  $\max\{e(A),e(B),e(C)\}\leq n^2/48+o(n^2)$. 
\end{proof}

\bibliographystyle{abbrvurl}
\bibliography{bibilo}

\begin{thebibliography}{10}

\bibitem{MR1983363}
N.~Alon, B.~Bollob\'{a}s, M.~Krivelevich, and B.~Sudakov.
\newblock Maximum cuts and judicious partitions in graphs without short cycles.
\newblock {\em J. Combin. Theory Ser. B}, 88(2):329--346, 2003.
\newblock \href {https://doi.org/10.1016/S0095-8956(03)00036-4}
  {\path{doi:10.1016/S0095-8956(03)00036-4}}.

\bibitem{flagbipartite}
J.~Balogh, F.~C. Clemen, and B.~Lidick\'{y}.
\newblock Max cuts in triangle-free graphs.
\newblock {\em Trends in Mathematics (Extended Abstracts EuroComb 2021)},
  14:509--514, 2021.
\newblock \href {https://doi.org/10.1007/978-3-030-83823-2_82}
  {\path{doi:10.1007/978-3-030-83823-2_82}}.

\bibitem{MR1945378}
B.~Bollob\'{a}s and A.~D. Scott.
\newblock Problems and results on judicious partitions.
\newblock volume~21, pages 414--430. 2002.
\newblock Random structures and algorithms (Poznan, 2001).
\newblock \href {https://doi.org/10.1002/rsa.10062}
  {\path{doi:10.1002/rsa.10062}}.

\bibitem{MR1606776}
S.~Brandt.
\newblock The local density of triangle-free graphs.
\newblock {\em Discrete Math.}, 183(1-3):17--25, 1998.
\newblock \href {https://doi.org/10.1016/S0012-365X(97)00074-5}
  {\path{doi:10.1016/S0012-365X(97)00074-5}}.

\bibitem{MR0360330}
V.~Chv\'{a}tal.
\newblock The minimality of the {M}ycielski graph.
\newblock In {\em Graphs and combinatorics ({P}roc. {C}apital {C}onf., {G}eorge
  {W}ashington {U}niv., {W}ashington, {D}.{C}., 1973)}, pages 243--246. Lecture
  Notes in Math., Vol. 406, 1974.

\bibitem{MR0409246}
P.~Erd\H{o}s.
\newblock Problems and results in graph theory and combinatorial analysis.
\newblock In {\em Proceedings of the {F}ifth {B}ritish {C}ombinatorial
  {C}onference ({U}niv. {A}berdeen, {A}berdeen, 1975)}, pages 169--192.
  Congressus Numerantium, No. XV, 1976.

\bibitem{MR1439273}
P.~Erd\H{o}s.
\newblock Some old and new problems in various branches of combinatorics.
\newblock {\em Discrete Math.}, 165/166:227--231, 1997.
\newblock \href {https://doi.org/10.1016/S0012-365X(96)00173-2}
  {\path{doi:10.1016/S0012-365X(96)00173-2}}.

\bibitem{Howtotriangle}
P.~Erd\H{o}s, R.~Faudree, J.~Pach, and J.~Spencer.
\newblock How to make a graph bipartite.
\newblock {\em J. Combin. Theory Ser. B}, 45(1):86--98, 1988.
\newblock \href {https://doi.org/10.1016/0095-8956(88)90057-3}
  {\path{doi:10.1016/0095-8956(88)90057-3}}.

\bibitem{MR1273598}
P.~Erd\H{o}s, R.~J. Faudree, C.~C. Rousseau, and R.~H. Schelp.
\newblock A local density condition for triangles.
\newblock {\em Discrete Math.}, 127:153--161, 1994.
\newblock Graph theory and applications (Hakone, 1990).
\newblock \href {https://doi.org/10.1016/0012-365X(92)00474-6}
  {\path{doi:10.1016/0012-365X(92)00474-6}}.

\bibitem{Howmany}
P.~Erd\H{o}s, E.~Gy\H{o}ri, and M.~Simonovits.
\newblock How many edges should be deleted to make a triangle-free graph
  bipartite?
\newblock In {\em Sets, graphs and numbers ({B}udapest, 1991)}, volume~60 of
  {\em Colloq. Math. Soc. J\'{a}nos Bolyai}, pages 239--263. North-Holland,
  Amsterdam, 1992.

\bibitem{MR2228942}
U.~Feige and M.~Langberg.
\newblock The {$\rm RPR^2$} rounding technique for semidefinite programs.
\newblock {\em J.~Algorithms}, 60(1):1--23, 2006.
\newblock \href {https://doi.org/10.1016/j.jalgor.2004.11.003}
  {\path{doi:10.1016/j.jalgor.2004.11.003}}.

\bibitem{MR1367967}
A.~Frieze and M.~Jerrum.
\newblock Improved approximation algorithms for {MAX} {$k$}-{CUT} and {MAX}
  {BISECTION}.
\newblock In {\em Integer programming and combinatorial optimization
  ({C}openhagen, 1995)}, volume 920 of {\em Lecture Notes in Comput. Sci.},
  pages 1--13. Springer, Berlin, 1995.
\newblock \href {https://doi.org/10.1007/3-540-59408-6_37}
  {\path{doi:10.1007/3-540-59408-6_37}}.

\bibitem{flagmulti}
P.~Hu, B.~Lidick\'{y}, T.~Martins, S.~Norin, and J.~Volec.
\newblock Large multipartite subgraphs in {$H$}-free graphs.
\newblock {\em Trends in Mathematics (Extended Abstracts EuroComb 2021)},
  14:707--713, 2021.
\newblock \href {https://doi.org/10.1007/978-3-030-83823-2_113}
  {\path{doi:10.1007/978-3-030-83823-2_113}}.

\bibitem{MR1967401}
P.~Keevash and B.~Sudakov.
\newblock Local density in graphs with forbidden subgraphs.
\newblock {\em Combin. Probab. Comput.}, 12(2):139--153, 2003.
\newblock \href {https://doi.org/10.1017/S0963548302005539}
  {\path{doi:10.1017/S0963548302005539}}.

\bibitem{MR2232396}
P.~Keevash and B.~Sudakov.
\newblock Sparse halves in triangle-free graphs.
\newblock {\em J. Combin. Theory Ser. B}, 96(4):614--620, 2006.
\newblock \href {https://doi.org/10.1016/j.jctb.2005.11.003}
  {\path{doi:10.1016/j.jctb.2005.11.003}}.

\bibitem{MR1320169}
M.~Krivelevich.
\newblock On the edge distribution in triangle-free graphs.
\newblock {\em J. Combin. Theory Ser. B}, 63(2):245--260, 1995.
\newblock \href {https://doi.org/10.1006/jctb.1995.1018}
  {\path{doi:10.1006/jctb.1995.1018}}.

\bibitem{MR3096334}
C.~Lee, P.-S. Loh, and B.~Sudakov.
\newblock Bisections of graphs.
\newblock {\em J. Combin. Theory Ser. B}, 103(5):599--629, 2013.
\newblock \href {https://doi.org/10.1016/j.jctb.2013.06.002}
  {\path{doi:10.1016/j.jctb.2013.06.002}}.

\bibitem{LiuMaK4}
X.~Liu and J.~Ma.
\newblock Sparse halves in {$K_4$}-free graphs.
\newblock {\em arXiv:2007.14623}, 2020.

\bibitem{navesgrwc}
H.~Naves.
\newblock A problem of {E}rd{\H{o}}s on triangle-free graphs, 2015.

\bibitem{MR3383248}
S.~Norin and L.~Yepremyan.
\newblock Sparse halves in dense triangle-free graphs.
\newblock {\em J. Combin. Theory Ser. B}, 115:1--25, 2015.
\newblock \href {https://doi.org/10.1016/j.jctb.2015.04.006}
  {\path{doi:10.1016/j.jctb.2015.04.006}}.

\bibitem{flagsRaz}
A.~A. Razborov.
\newblock Flag algebras.
\newblock {\em J. Symbolic Logic}, 72(4):1239--1282, 2007.
\newblock \href {https://doi.org/10.2178/jsl/1203350785}
  {\path{doi:10.2178/jsl/1203350785}}.

\bibitem{sparsehalfraz}
A.~A. Razborov.
\newblock More about sparse halves in triangle-free graphs.
\newblock {\em arXiv:2104.09406}, 2021.

\bibitem{ReiherK4}
C.~Reiher.
\newblock {$K_4$}-free graphs have sparse halves.
\newblock {\em arXiv:2108.07297}, 2021.

\bibitem{MR2359832}
B.~Sudakov.
\newblock Making a {$K_4$}-free graph bipartite.
\newblock {\em Combinatorica}, 27(4):509--518, 2007.
\newblock \href {https://doi.org/10.1007/s00493-007-2238-0}
  {\path{doi:10.1007/s00493-007-2238-0}}.

\bibitem{Turanstheorem}
P.~Tur\'{a}n.
\newblock Eine {E}xtremalaufgabe aus der {G}raphentheorie.
\newblock {\em Mat. Fiz. Lapok}, 48:436--452, 1941.

\end{thebibliography}

\end{document}